\documentclass[a4paper]{amsart} %,11pt

\usepackage[colorlinks,linkcolor=blue,citecolor=blue,urlcolor=blue]{hyperref}
\usepackage[english]{babel}
\usepackage{amsthm,amsfonts,amssymb,mathrsfs,amsmath}
\usepackage{enumitem}%[shortlabels]{enumitem}
\usepackage{xcolor}
\usepackage{soul}
\usepackage{algorithm}
\usepackage{algpseudocode}
\usepackage{mathtools}						% uncomment to tag only the referred equations (post-production)
\mathtoolsset{showonlyrefs,showmanualtags}	% uncomment to tag only the referred equations (post-production)

\usepackage{algorithm,setspace}
\usepackage{algpseudocode}
\usepackage{graphics}
\usepackage{epsfig}
\usepackage{epstopdf}

\usepackage{booktabs}       % professional-quality tables
\usepackage{amsthm,amsfonts,amssymb,mathrsfs,amsmath}%\usepackage{amsfonts}       % blackboard math symbols
\usepackage{nicefrac}       % compact symbols for 1/2, etc.
\usepackage{microtype}      % microtypography
\usepackage{xcolor}         % colors

\usepackage{multicol}
\usepackage{mathtools}	

\newcommand{\R}{\mathbb{R}}

\definecolor{boris}{rgb}{0.,0.,0.}%{0.9,0.0,0.9}     

\definecolor{sofya}{rgb}{0.,0.,0.}%{0.0, 0.7, 0.3}
\def\sofya{\textcolor{sofya}}

\definecolor{chris}{rgb}{0.,0.,0.}%{0, 0., 1}
\def\chris{\textcolor{chris}}

\definecolor{pranav}{rgb}{0.,0.,0.}%{0.6, 0.1, 0.6}

\newtheorem{Theorem}{Theorem}[section]

\theoremstyle{definition}

\newtheorem{remark}[Theorem]{Remark}
\newtheorem{example}[Theorem]{Example}
\numberwithin{equation}{section}

\title[Higher order reversible integrators in deep learning]{Adaptive higher order reversible integrators for memory efficient deep learning}

\author{S.\ Maslovskaya} \address{Sofya Maslovskaya\\ Paderborn University, Department of Mathematics, Warburger Str. 100, 33098 Paderborn, Germany}
\email{sofyam@math.upb.de}
\author{S.\ Ober-Blöbaum} \address{Sina Ober-Blöbaum \\ Paderborn University, Department of Mathematics, Warburger Str. 100, 33098 Paderborn, Germany} \email{sinaober@math.upb.de}
\author{C.\ Offen} \address{Christian Offen\\ Paderborn University, Department of Mathematics, Warburger Str. 100, 33098 Paderborn, Germany} \email{christian.offen@upb.de}

\author{P.\ Singh} \address{Pranav Singh \\ Department of Mathematical Sciences, University of Bath, Bath, BA2 7AY, United Kingdom} \email{ps2106@bath.ac.uk}

\author{B.\ Wembe} \address{Boris Wembe \\ Paderborn University, Department of Mathematics, Warburger Str. 100, 33098 Paderborn, Germany} \email{wboris@math.upb.de}

\begin{document}

	\maketitle
	
	%\chris{Title: Do we want to change "methods" to "integrators"? }
	
	\begin{abstract}
		The depth of networks plays a crucial role in the effectiveness of deep learning. However, the memory requirement for backpropagation scales linearly with the number of layers, which leads to memory bottlenecks during training. Moreover, deep networks are often unable to handle time-series \textcolor{black}{data} appearing at irregular intervals. These issues can be resolved by considering continuous-depth networks based on the neural ODE framework in combination with reversible integration methods that allow for variable time-steps. Reversibility of the method ensures that the memory requirement for training is independent of network depth, while variable time-steps are required for assimilating time-series data on irregular intervals. However, at present, there are no known higher-order reversible methods with this property. High-order methods are especially important when a high level of accuracy in learning is required or when small time-steps are necessary due to large errors in time integration of neural ODEs, for instance in context of complex dynamical systems such as Kepler systems and molecular dynamics. The requirement of small time-steps when using a low-order method can significantly increase the computational cost of training as well as inference. In this work, we present an approach for constructing high-order reversible methods that allow adaptive time-stepping. Our numerical tests show the advantages in computational speed when applied to the task of learning dynamical systems.
	\end{abstract}
	
	%\chris{"Maybe we can shorten the abstract?"}
	
%\vspace{-5mm}
\section{Introduction}
Deep neural networks are widely used across various learning tasks \cite{Russakovsky2015,Esteva2017}, and their depth often plays a crucial role in the effectiveness of learning. These networks have also been shown to be particularly useful in the tasks of learning models of dynamical systems \cite{Chen2018,raissi2018m,Rudy2019,9003133,RUDY2019483,Liu2022,raissi2018multistep}. It was recently shown that the use of numerical methods for neural network architectures can provide impressive  results with theoretical guarantees \cite{Haber2017,chang2018,CELLEDONI2021,maslovskaya2024}.
In this work we use the theory of symmetric numerical methods for the construction of a new class of reversible neural networks. The new class allows memory efficient computations of gradients in training and reduced computational costs in learning {models of} dynamical systems,  \textcolor{black}{where the parameters typically need to be identified \textcolor{black}{to} high accuracy and the depth of the networks can be very large.} \textcolor{black}{Our network architecture constitutes an important contribution to ensure scalability of neural ODEs to high-dimensional dynamical systems that arise, for instance, as discretizations of systems governed by partial differential equations.} 

The high memory costs of computing the gradient of very deep neural networks using the backpropagation algorithms poses a significant bottleneck in their training, hindering their scalability and efficiency. 
To address this, a neural ODE approach combined with the adjoint method has been proposed for gradient computations \cite{Chen2018}, which avoids storing intermediate states during forward propagation, potentially making the cost of gradient computation independent of network depth. However, it was quickly realized  \cite{Gholaminejad2020,Zhuang2020} that using this approach with arbitrary discretization methods leads to incorrect gradients.

%Deep neural networks are used efficiently in various learning tasks \cite{Russakovsky2015,Esteva2017}. In particular, they were shown to be useful in the tasks of learning {models of} dynamical systems \cite{9003133}. However, memory costs become an obstacle for using very deep neural networks. It was recently shown that the use of numerical methods for neural network architectures can provide impressive  results with theoretical guarantees \cite{maslovskaya2024,CELLEDONI2021,chang2018,Haber2017}. In this work we use the theory of symmetric numerical methods for the construction of a new class of reversible neural networks. The new class allows memory efficient computations of gradients in training and reduced computational costs in learning {models of} dynamical systems. 
%\chris{for which gradients can be computed in a memory efficient way for training.}
%which are memory efficient and provide correct gradients.
%In addition, this new class shows advantages in learning dynamics systems.

%High memory costs were recognized as a major bottleneck in deep learning and in \cite{Chen2018} it was suggested to use a neural ODE approach together with the adjoint method for gradient computations, which avoids storing of intermediate states obtained during forward propagation in the network. It was quickly realized \cite{Zhuang2020} that this method  provides wrong gradients when used with an arbitrary discretization method. The solution of this problem is to use reversible integrators. %One of the ways to construct the reversible integrators is to use symplectic methods 

The solution of this problem is to use reversible integrators, which ensure accurate gradient computations by reconstructing intermediate states precisely during backward integration. However, 
symplectic reversible integrators \cite{chang2018}, which are a large class of well studied integrators, %, but such methods 
do not allow adaptivity in the step size and require a particular structure in the neural ODE. This makes them unsuitable for use in time-series applications where data appears at irregular intervals, when time-steps need to be decreased adaptively to achieve a prescribed accuracy in the learning of dynamical systems, \textcolor{black}{or when the identified model is used to predict continuous trajectories.}

\textcolor{black}{Adaptive time-stepping for numerical integrators for differential equations is a well-established field in numerical analysis \cite{Hairer2013,Deuflhard2002}. \textcolor{black}{In adaptive time-stepping, s}tep sizes for an integration step are selected such that an estimate for the local error is below a given error tolerance. In this way, computational cost in numerical integration can be saved when large step sizes are sufficient to obtain accurate results, while step sizes are automatically decreased when required.}
%\textcolor{black}{The reversible integrators, such as Verlet method \cite{Verlet1967} do not allow adaptive time-stepping in general and are only applicable to particular classes of differential equations.}
The only reversible methods \textcolor{black}{compatible with} variable time-step \textcolor{black}{selection without losing \textcolor{black}{the} reversibility property} are asynchronous leapfrog (ALF) \cite{mutze2016asynchronous}, \textcolor{black}{which is based on the classical Verlet method (also known as leapfrog method) \cite{Verlet1967}}, and the reversible Heun method \cite{Kidger2021a}.

\textcolor{black}{ALF has been used to construct neural network architectures known as MALI networks \cite{zhuang2021mali}.}
These methods are based on operating on an augmented space: a neural ODE in the original variable $z$ 
is extended to a larger space and replaced by a neural ODE in $(z,v)$. Both methods are known to be of order of accuracy $(2,1)$ in $(z,v)$ \cite{mutze2016asynchronous}, which makes them computationally costly in learning tasks that require high accuracy in the integration of the neural ODE, in particular, in learning of dynamical systems.
%In such tasks the learning accuracy is of high value.
To be able to reach high accuracy, the lower order methods are forced to use small step sizes and, as a result, \textcolor{black}{have} higher computational costs. This phenomenon was highlighted in the examples in \cite{Matsubara2021}.
\textcolor{black}{Furthermore, we show in \ref{sec:Kepler} that in parameter identification tasks there is a direct relation between the order of a numerical integrator and the order of accuracy of identified parameters.}
Therefore, there is a need for higher order reversible methods, which we address in this paper. %\ps{[you 
	%could mention here that the existing reversible integrators that allow the use of variable time-steps are based on operating on an augmented space: a neural ODE in the original variable $z$ 
	%is extended to a larger space and replaced by a neural ODE in $(z,v)$. ]}. Currently, the existing ones are asynchronous leapfrog \cite{zhuang2021mali} and reversible Heun \cite{Kidger2021a}, both are methods of order 2 \ps{[They are known to be $(2,1)$ in $(z,v)$, which is unusual of course. In fact, you show that ALF is order $2$ in both when two steps are used. We also should mention that we found out about Kidger, Foster, Lyon etc paper (ending with `neural SDEs' regarding reversible Heun having such a property), but did so near the submission of the paper. I am not sure if all these details need to be here, but they need to appear somewhere.]}. Another option is to use symplectic methods \cite{chang2018}, but they do not allow adaptivity in step size and require a particular structure in the neural ODE. 

\textcolor{black}{
	Various machine learning techniques \textcolor{black}{have} emerged in the past decades for approximating models of dynamical systems \cite{Ghadami2022}. %The \textcolor{black}{proposed} approaches 
	These include methods based on Gaussian processes \cite{Bouvrie2017,Raissi2018,Hamzi2021}, sparse regression on libraries of basis functions \cite{Brunton2016,Tran2017,Reinbold2021}, %The training data is typically represented by time series of observed trajectories and the learning task can be formulated differently depending on the underlying motivation, e.g., learning of an unknown vector field, a generating function of the vector field or the flow. 
	% One of the learning approaches is based on neural networks (NNs), where the 
	and recurrent neural networks (RNNs) %have the most suitable structure for the treatment of time series, as they mimic the flow of a dynamical system 
	\cite{Funahashi1993,BailerJones1998,Karniadakis2021}. In particular, the neural ODE approach \cite{Chen2018} identified an important connection between %permitted to combine 
	the RNN structure and the numerical methods available for integration of differential equations. This was generalized by universal differential equations in \cite{rackauckas2021} for different types of differential equations. Other generalizations include physics informed learning of dynamical systems \cite{Greydanus2019,LNN,Jin2020,Chen2020Symplectic} and operator approximation \cite{Chen1995,Lu2021,Lin2023,Boulle2024}. If the training data consists of time-series data that corresponds to a constant time-step, neural ODEs can be trained with a low order method and time-series can be predicted with high accuracy provided that the trained neural ODE is integrated with the same integrator that was used during training \cite{zhu2021inverse,DAVID2023112495,Offen2022,Offen2023}. However, in realistic examples, snapshots of trajectories with variable time-steps need to be processed \cite{raissi2018multistep,Rudy2019,Liu2022}. Moreover, a discretization-independent prediction of the system's evolution is often desired and learning of underlying differential equations requires high accuracy in the learned parameters, which requires high accuracy simulation.}

\textcolor{black}{We demonstrate that our approach can be applied to high-dimensional dynamical systems, including those arising from discretizations of  partial differential equations. This is a highly active research area. Other approaches in this context include model order reduction based techniques and operator inference (see the review \cite{Kramer2024} or e.g.~\cite{SHARMA2024116865,SHARMA2023116402,Blanchette2020,Blanchette2022}), or structure-preserving approaches for discrete field theories \cite{Qin2020,Offen2024,offen2024machinelearningdiscretefield}.
}

The main contribution of this paper is the development of a methodology to construct reversible neural networks based on higher order numerical methods. First, we prove that, in contrast to the analysis in \cite{mutze2016asynchronous},  ALF
%has a particular property to be 
\chris{is} of order 2 in both $(z,v)$ at even time-steps. Using this property and the theory of composition methods, we construct a class of reversible networks of any even order.
A particular architecture based on 4th order networks is compared to the already known ALF method
%applied to 
\chris{on} examples of learning dynamical systems. The comparison of the two methods with adaptive time-stepping shows that the proposed higher order method is computationally more efficient. % \chris{[What do you mean with "time-efficient"? Do you mean "computationally efficient" or that they train faster?]}. 
%On the other hand, comparing the methods with fixed step size we observe that the higher order method is able to reach a higher training accuracy in the learned dynamics and does not require the use of very small time-steps to reach it. %Notice that the reversible Heun's method was observed to have the property to be of order 2 in both $(z,v)$ at even time-steps, as reported in \cite{Kidger2021a}, which makes it possible to use our methodology to the construction of higher order methods based reversible Heun as well.
%\vspace{-4mm}
%
\section{Background}
%\vspace{-2mm}
%
\subsection{Neural ODE}
%\vspace{-2mm}
Assume that an unknown function $\mathcal{F} : X \rightarrow Y$ is approximated by a neural network based on training data $\{x_i,  y_i = \mathcal{F}(x_i) \}^{n}_{i = 1}$. %The neural network architecture is given \chris{as a} composition of parameterized nonlinear activation functions $\sigma$ and the propagation in a feature space $\R^n$ is defined by  \begin{equation} \label{eq:neural.ODE.0}
	%\begin{aligned}
	%&  z_{j+1}=\sigma (z_{j}, \theta_j), \quad j = 0, \dots, N-1, \\
	%& z_0 = z,
	%\end{aligned}
	%\end{equation}
	%where $z_j$ is the value of the feature variable at the $j$th layer with the total number of layers $N$ and $z = (x_1, \dots, x_n)$.
	%Such kind of networks can be constructed based on discretized differential equations. 
	\textcolor{black}{The neural ODE approach to the deep network design employs the idea of continuous-depth networks and their discretization by \textcolor{black}{a} numerical method. The continuous-depth network is defined as a flow of \textcolor{black}{a} neural ODE of the form} %In this case the differential equation is called neural ordinary differential equation (neural ODE).  Then, $z_0, \dots, z_N$ satisfying \eqref{eq:neural.ODE} approximate a flow of a neural ODE given by
	\begin{equation} \label{eq:ODE} 
		\dot z(t) = f(z(t), \theta(t)), \quad z(0) = (x_1, \dots, x_n),
	\end{equation}
	on \textcolor{black}{the} time interval $[0,T]$ for some vector field $f$. \textcolor{black}{Discretization of a neural ODE with a numerical method of step size  $h$ is defined as follows}
	\begin{equation} \label{eq:neural.ODE}
		\begin{aligned}
			&  z_{j+1}=\sigma_h (z_{j}, \theta_j), \quad j = 0, \dots, N-1, \\
			& z_0 =  z,%\textcolor{black}{(x_1, \dots, x_n)},
		\end{aligned}
	\end{equation}
	\textcolor{black}{where $z_j$ is the value of the feature variable at the $j$th layer with the total number of layers $N$ and $z = (x_1, \dots, x_n)$.  
		The step size can be chosen in an adaptive manner for each layer and it is not considered as an optimization parameter.} % !!!Here a paragraph on the adaptive time-stepping procedure!!!}
%\textcolor{black}{where $\theta_j$ parameters of the network.} %The properties of the numerical flow of such a differential equations are then inherited by the neural network. 
The learning problem is to find parameters $\{\theta_j\}_{j=0}^{N-1}$ % in the neural network \eqref{eq:neural.ODE} 
which lead to the best approximation of $\mathcal{F}$. % by the corresponding neural network. 
The parameters are usually found as a solution of the following optimization problem. %of an optimization problem of the following form. %To find the parameters $\{W_i, \beta_i\}$ for which the network \eqref{eq:neural.ODE} approximate $F$ in the best possible way, an optimization problem is formulated in the following form.
\begin{equation} \label{eq:DL:optimal.problem}
	\begin{aligned}
		& \min_{\{ \theta_j \}} J =  L(z_N, y) \\
		&  z_{j+1}=\sigma_{\textcolor{black}{h}} (z_{j}, \theta_j), \quad j = 0, \dots, N-1, \\
		& z_0 = z,
	\end{aligned}
\end{equation}
where %$\Phi(\cdot)$ is the output activation function which is defined on the space of features $x$ with its image in the space of labels $y$, 
$L(\cdot, \cdot)$ is a loss function which measures the distance between the output of the network and %the corresponding labels of the training data 
the training data \textcolor{black}{$y = (y_1, \dots, y_N)$}. 

Because of the connection between network \eqref{eq:neural.ODE} and the corresponding neural ODE \eqref{eq:ODE}, there exists a continuous counterpart of \eqref{eq:DL:optimal.problem} which makes \eqref{eq:DL:optimal.problem} an approximation of an %infinite dimensional optimization problem, more precisely, an 
optimal control problem of the form
\begin{equation}\label{eq:OCP}
	\begin{aligned}
		& \min_{\theta(t)} J =  L(z(T), y)  \\
		&  \dot  z(t) = f(z(t), \theta(t)), \quad t \in [0,T],\\
		&  z(0) = z.
	\end{aligned}
\end{equation}  			
%\vspace{-3mm}
%\textcolor{black}{SO: $W,\beta$ not defined, probably only use $\theta$} \sofya{Modified to $\theta$ everywhere}

%\subsection{Methods of gradient computations}
Solutions of %the optimization problem 
\eqref{eq:DL:optimal.problem} are usually found using methods based on gradient descent. Such methods require computation\chris{s} of the gradients of the loss \chris{function} with respect to all parameters \textcolor{black}{$\{\theta_j\}_{j=0}^{N-1}$}.% for $j = 1, \dotsm N-1$. 
%%
%
%There are different ways to compute gradients, but all of them are based on the backpropagation formula
%The gradients are calculated by backpropagation according to the chain rule
%
% \chris{By the chain rule }
%  $$ \frac{\partial}{\partial \theta_{j}} J = \nabla_z L(z_{N}, y)\chris{^\top} \frac{\partial z_{N} }{\partial z_{N-1}} \cdots \frac{\partial z_{j+2} }{\partial z_{j+1}} \frac{\partial z_{j+1} }{\partial \theta_{j}}.
% $$
% $$ \frac{\partial}{\partial \theta_{j}} J(\theta) = \nabla_z L(z_{N}(\theta))\chris{^\top} \frac{\partial z_{N}(\theta) }{\partial z_{N-1}(\theta)} \cdots \frac{\partial z_{j+1}(\theta) }{\partial z_{j}(\theta)} \frac{\partial z_{j}(\theta) }{\partial \theta_{j}}.
% $$
% \chris{To avoid computationally expensive multiplication of large matrices, the formula is evaluated from the left to the right (backpropagation). This requires that the intermediate values $z_j$ ($j=1,\ldots,N$) are available. Their size corresponds to the width of the layer and their number to the network's depth~$N$.} %While $z_j$ can be stored during the forward pass, in the following we discuss memory efficient methods for deep neural networks.}
%
%%\vspace{-2mm}
\subsection{Methods of gradient computations}
%There are two main approaches to calculate the gradients in deep learning, namely, the approach called discretize-then-optimise and the approach called optimise-then-discretize. The main difference is that in the first approach, the gradients are calculated with respect to the discrete flow \eqref{eq:DL:optimal.problem} and in the second approach the gradients
%\chris{relate to the continuous problem \eqref{eq:OCP}.}
%are obtained by discretising the continuous gradients.
%The first approach is the default approach in deep learning. Its problem is in the memory requirement, which is resolved in the second approach at the cost of restrictions on the network structure.
%\vspace{-8mm}
\textcolor{black}{\paragraph{Backpropagation} Let $J =  L(z_N, y)$. By the chain rule,} \textcolor{black}{the gradient is given by} 
\textcolor{black}{$$ \frac{\partial}{\partial \theta_{j}} J = \nabla_z L(z_{N}, y)^\top \frac{\partial z_{N} }{\partial z_{N-1}} \cdots \frac{\partial z_{j+2} }{\partial z_{j+1}} \frac{\partial z_{j+1} }{\partial \theta_{j}}.
$$
% $$ \frac{\partial}{\partial \theta_{j}} J(\theta) = \nabla_z L(z_{N}(\theta))\chris{^\top} \frac{\partial z_{N}(\theta) }{\partial z_{N-1}(\theta)} \cdots \frac{\partial z_{j+1}(\theta) }{\partial z_{j}(\theta)} \frac{\partial z_{j}(\theta) }{\partial \theta_{j}}.
% $$
To avoid computationally expensive multiplication of large matrices, the formula is evaluated from the left to the right (backpropagation). This requires that the intermediate values $z_j$ ($j=1,\ldots,N$) are available. Their size corresponds to the width of the layer and their number to the network's depth~$N$.} %\textcolor{black}{By construction, all the gradients are computed exactly. Notice that when the adaptive stepping in the discretiztaion is used, the memory costs are especially high because the computational graph is saved also for the computaton of the step size.\cite{}}%The calculation of the gradients in this approach is done in two steps. First step is to propagate \chris{forward} through the network $z_{j} \rightarrow z_{j+1}$ using the current values of parameters $\theta_j$ determined in the previous iteration of the optimization algorithm or by initialisation. All the values of $z_{j}$ and the computational graph are saved to be accessed next in the backpropagation step. In this step the gradient $\nabla_{\theta_j} J$ is calculated using the reverse mode of automatic differentiation (AD). By construction, all the gradients are computed exactly. \chris{T}herefore, the obtained gradients are exact for the discrete optimization problem. \textcolor{black}{Notice that when the adaptive stepping in the discretiztaion is used, the memory costs are especially high because the computational graph is saved}
%\vspace{-2mm}
\paragraph{Adjoint method}
This approach is based on the formula for the continuous gradients via adjoint variables $p(t)$. The adjoint variables are the solution of 
\begin{equation} \label{eq:adjiont}
\dot p = - \frac{\partial}{\partial z} f(z(t), \theta(t))^\top p, \qquad p(T) = \nabla L(z(T),y),
\end{equation}
on time interval $[0,T]$ and the differential of $J(\theta)$ with respect to $\theta(t)$ is calculated as follows
\begin{equation} \label{eq:adjiont.grad} D J(\theta(t)) = p^\top(t) \frac{\partial}{\partial \theta} f(z(t), \theta(t)).
\end{equation}
This method is particularly interesting when the available memory is limited. In this case, the forward propagation is done to obtain the value of $z(T)$ and $\nabla L(z(T),y)$, there is no need to save the intermediate values $z_{j}$ and the computational graph, because the backpropagation is realized by integrating the equations \chris{for} $(z(t), p(t))$ numerically backward in time. The correct discretization of the state-adjoint equations leads to the same expression for the gradient as the one obtained in the \textcolor{black}{backpropagation} approach. Still, the values of $z_j$ obtained by the numerical integration backward do not always coincide with the values obtained in the forward pass. This is why this approach usually leads to inexact gradients. This issue can be solved by considering reversible networks. % together with the correct discretisation of the adjoint equation. 
%\vspace{-3mm}
\textcolor{black}{\paragraph{Checkpointing} \textcolor{black}{An} alternative approach for memory reduction is checkpointing \cite{Gholaminejad2020}, which \textcolor{black}{stores a few} intermediate states for a regeneration of the computation graph. In case of adaptive time-stepping it was described in \cite{Zhuang2020} and further improved in \cite{Matsubara2021} for the class of Runge-Kutta methods. This is a highly efficient approach, when used in combination with higher order methods. In this case, a small number of checkpoints is required to get a high accuracy in learning. However, when the learning task is to learn a dynamical system from long trajectories of complex systems, then the number of checkpoints becomes large and can lead to memory leaks. This is why it is important to have an alternative approach based on reversible networks with the memory costs independent from the given task.}
%\vspace{-2mm}
\subsection{Reversible neural network}
%\vspace{-2mm}
A reversible network is a network with the property that there exists an explicit formula for backward propagation $z_{j+1} \mapsto z_j$ that exactly inverts a forward pass $z_{j} \mapsto z_{j+1}$\textcolor{black}{, i.e., there exists a map $\tilde \sigma(\cdot)$, such that $z_{j} = \tilde \sigma(z_{j+1})$ and $z_{j+1} = \sigma_h (\tilde \sigma(z_{j+1}), \theta_j)$}.  It requires that the time-steps $t_0,\ldots, t_N$ have been stored when the forward propagation was computed but it does not require storage of the (potentially very high-dimensional) intermediate values $z_j$. The discretizations of neural ODE \eqref{eq:neural.ODE}, which admit this property are called reversible methods. As for now, there are only two known reversible methods allowing for an adaptive choice of the step size, namely, asynchronous leapfrog \cite{zhuang2021mali} and reversible Heun \cite{Kidger2021a}.  

This notion of reversibility \chris{for neural networks} needs to be contrasted \chris{with} the notion of time-reversibility or symmetry \chris{for numerical integrators}.
\chris{In the context of neural networks, reversibility means that there exists an explicit, efficient formula to invert the forward pass. In numerical integration theory, a time-reversible or symmetric numerical integrator is a formula to advance the solution of an ordinary differential equation by time $h$ such that its inverse is obtained by substituting $h$ by $-h$} \cite[II.3]{hairer2006}. \textcolor{black}{In case of the dynamical system $f(z(t), \theta(t))$ from \eqref{eq:neural.ODE}, if the discretization by a numerical method $z_{j+1} = \sigma_h(z_j, \theta(t_j))$ is symmetric, then it implies $z_j = \sigma_{-h}(z_{j+1}, \theta(t_{j+1}))$, %In the neural ODE context $\theta(t_{j+1})$ corresponds to the optimization parameter $\theta_{j+1}$. As a result, using a symmetric method in \eqref{eq:neural.ODE}, we obtain
or equivalently, $z_{j} = \sigma_{-h} (z_{j+1}, \theta_{j+1})$. Therefore, we can set $\tilde \sigma(\cdot) = \sigma_{-h} (\cdot, \theta_{j+1})$, which implies that the method is reversible.} The symmetry \chris{of integrators} is beneficial in the context of the article as inverses of the methods required for backpropagation take simple forms and efficient classical techniques to construct higher order methods \cite[II.4]{hairer2006} apply.
% which makes \eqref{eq:ODE} a non autonomous differential equations. Therefore, the numerical methods to discretize $f_\theta(z(t),t)$ can be stated for a general non autonomous system. 
%\vspace{-3mm}
\paragraph{Asynchronous Leapfrog (ALF) method}  %Consider \eqref{eq:ODE}.  
As the optimization parameter $\theta(t)$ in the dynamics $f(z(t),\theta(t))$ depends on time, it can be seen as a part of $f$ and written simply $f_\theta(z(t),t)$. The ALF method requires the augmentation of the pair of state and time $(z,t)$ with the velocity $v$ which approximates $f_\theta(z(t),t)$. \textcolor{black}{We denote a step forward of \textcolor{black}{the} ALF method with the step size $h$ by $\Psi^{ALF}_h$.} Given a triple $(z_j,v_j,t_j)$ and a step size $h$, the algorithm generates in the forward pass the next values  $(z_{j+1},v_{j+1},t_{j+1})$ as follows
\begin{equation} \label{eq:ALF.forward}
\begin{pmatrix} z_{j+1} \\[0.3cm] v_{j+1} \end{pmatrix} =\textcolor{black}{\Psi^{ALF}_h(z_{j},v_{j}, t_j) = } \begin{pmatrix} z_{j} + h f_\theta(z_{j} + \frac{h}{2}v_{j}, t_j + \frac{h}{2}) \\[0.3cm] 2f_\theta(z_{j} + \frac{h}{2}v_{j}, t_j + \frac{h}{2}) - v_{j} \end{pmatrix}, \quad t_{j+1}= t_j + h.   
\end{equation}
The step backward calculates $(z_j,v_j,t_j)$ from $(z_{j+1},v_{j+1},t_{j+1})$ as follows
\begin{equation} \label{eq:ALF.backward}
\begin{pmatrix} z_{j} \\[0.3cm] v_{j} \end{pmatrix} =\textcolor{black}{\Psi^{ALF}_{-h}(z_{j+1},v_{j+1}, t_{j+1}) }%\begin{pmatrix} z_{j+1} - h f(z_{j+1} - \frac{h}{2}v_{j+1}, t_{j+1} - \frac{h}{2}) \\[0.3cm] 2f(z_{j+1} - \frac{h}{2}v_{j+1}, t_{j+1} - \frac{h}{2}) - v_{j+1} \end{pmatrix}
,\quad t_j = t_{j+1} - h.
\end{equation}
%To see that the backward method leads exactly to the values calculated in the forward pass, i.e., the composition of an iteration forward and an iteration backward is an identity map, it is helpful to write it as follows
%$$ \begin{pmatrix} z^{k+1} \\[0.3cm] v^{k+1} \end{pmatrix} = \begin{pmatrix} z^{k} + h f(z^{k} + \frac{h}{2}v^{k}, t^k + \frac{h}{2}) \\[0.3cm] 2f(z^{k} + \frac{h}{2}v^{k}, t^k + \frac{h}{2}) - v^{k} \end{pmatrix}, \qquad 
%\begin{pmatrix} z^{k} \\[0.3cm] v^{k} \end{pmatrix} = \begin{pmatrix} z^{k+1} - h f(z^{k+1} - \frac{h}{2}v^{k+1}, t^{k+1} - \frac{h}{2}) \\[0.3cm] 2f(z^{k+1} - \frac{h}{2}v^{k+1}, t^{k+1} - \frac{h}{2}) - v^{k+1} \end{pmatrix}.$$
%or it can also be represented as follows
% $$ z^{k+1} = z^k + h v^k + \frac{h^2}{2} \frac{v_1 - v^k}{\tau} \qquad v^{k+1} = v^k + h \frac{v_1 - v^k}{\tau},$$
% where $\tau = h/2$. Let us denote the ALF step as a function of $h$ by $\mathcal{A}_h$. Then $\mathcal{A}_h$ can be seen as a product of 3 following maps. First $B_h(t, z, v) = (t+h, z+ hv, v)$, second $C(t,z,v) = (t,z,2f(t,z) - v)$. Then we have
% $$ {A}_h = B_{h/2} \circ C \circ  B_{h/2}.$$
%Notice also the similarity with the standard explicit midpoint rule defined by
% $$ z^{k+1} = z^k + hf(z^k + \frac{h}{2} f(z^k, t^k), t^k + \frac{h}{2}).$$
If the method is initialized at $(z_0,f(z_0,t_0),t_0)$, then ALF is a second order method in $z$ and first order method in $v$, as it was shown in \cite{zhuang2021mali}. The order of accuracy is the order in step size $h$ of the error of the numerical flow \textcolor{black}{compared} with the exact flow of the ODE \cite{hairer2006}.
%The ODE in this case is $\dot z(t) = f_\theta(z(t),t)$ and its flow is extended to $(z(t), v(t) = f_\theta(z(t),t))$.
\textcolor{black}{Notice that ALF is symmetric by definition.}%Notice that ALF has the property that the step backward\chris{s} coincides with the step forward applied to $-h$ instead of $h$, which makes this method symmetric.

% \paragraph{Reversible Heun method}~
The reversible Heun method is another reversible method \chris{based on state-space augmentation}
%follows the similar idea of state augmentation and
and was introduced in \cite{Kidger2021a}. %The state is augmented to $(v,z,\hat z, t)$ with initialization at $(f_\theta(z_0, t_0),z_0, z_0, t_0)$. %The step forward is defined as
% $$ t_{j+1} = t_j + h, \ \  \hat{z}_{k+1} = 2z_j - \hat{z}_{k}  + h v_j,\ \  v_{j+1} = f(t_{j+1}, \hat{z}_{k+1}), \ \  z_{j+1} = z_j + \frac{h}{2}(v_{j} + v_{j+1})
%$$
%and the step backward is defined by
% $$ t_{j} = t_{j+1} - h, \ \  \hat{z}_{k} = 2z_{j+1} - \hat{z}_{k+1}  - h v_{j+1},\ \  v_{j} = f(t^{k}, \hat{z}^{k}), \ \  z_{j} = z_{j+1} - \frac{h}{2}(v_{j} + v_{j+1})
%$$
% Standard Heun's method is defined by
% $$ 
% \begin{aligned}
% & \hat{z}^{k+1} =  z^k + hf(z^k,t^k), \\
% &  {z}^{k+1} = z^k + \frac{h}{2}\left( f(z^k,t^k) + f(\hat{z}^{k+1}, t^k + h) \right).
% \end{aligned}
% $$
The method was shown to be also of second order in $z$ and first order in $v$. In addition, it is a symmetric method. In the following part of the paper we will concentrate on the construction of higher order methods based on ALF, but the same can be also applied to the reversible Heun method.
%
%\vspace{-2mm}
\section{New reversible architectures}
%\vspace{-2mm}
%\subsection{Higher order reversible methods}
%
A general approach in numerical analysis to construct higher order symmetric methods is by composition \cite{hairer2006,Hairer2013,BlanesCasasMurua2024}. In this case one can start with a lower order numerical method and construct a new method by composition of the lower order method with a particular choice of step sizes. This construction leads to a method of higher order of accuracy.  %However, there is no theory for the construction of a higher order methods based on the composition of mixed-order methods, such as ALF, which is of order $(2,1)$ in variables $(z,v)$. %Even though the order of consistency of ALF in $v$ is only 1 \cite{zhuang2021mali}, testing ALF on a numerical example (see below), we observe second order convergence in $(z,v)$ (Figure~\ref{fig:ALF}). We will show that a method consisting of a composition of two steps of ALF has order of consistency 2 in $(z,v)$, which explains the convergence behaviour. 
%
%\chris{Testing ALF on a numerical example (see below), we observe second order convergence in $(z,v)$ (Figure~\ref{fig:ALF}). This is surprising as the order of consistency of ALF in $v$ is only 1 \cite{zhuang2021mali}. Indeed, we will show that a method consisting of a composition of two steps of ALF has order of consistency 2 in $(z,v)$, which explains the convergence behaviour.}
%By considering the evolution of the numerical flow of ALF in time, one easily notices that the method is of second order in both $(z,v)$, as shown in Figure~\ref{fig:ALF} \chris{for a numerical example}. The explanation is that the method gains order $2$ in every second step. So that the values obtained at $(z_{2k},v_{2k})$ are of the second order accuracy with respect to the exact flow. Let us consider a simple example and the order of the global error of the numerical trajectories obtained by ALF.
%\textcolor{black}{SO: Not totally clear what is meant here.} 
%\sofya{There is no theory for composition of mixed order methods.}
%Solving the equation numerically using ALF method leads 

\textcolor{black}{Numerical experiments, see Appendix~\ref{sec:ALF2}, show that ALF} \textcolor{black}{is of} second order in the error with respect to a high accuracy solver in both $z$ and $v$. This is surprising as the order of consistency of ALF in $v$ is only 1 \cite{zhuang2021mali}. Indeed, as we show below, a method consisting of a composition of two steps of ALF has order of consistency 2 in $(z,v)$, which explains the convergence behaviour.
%Since there is no theory for composition of mixed order methods, 
This observation \textcolor{black}{is required to apply theory for} composition methods \cite{Hairer2013,BlanesCasasMurua2024,Yoshida1990} to (two steps of) ALF.
%The result on the order of ALF is summarized in the following theorem. 
\begin{Theorem} \label{th:ALF2}
Composition of two steps of ALF methods, i.e. \textcolor{black}{$\Psi^{ALF}_{h/2}\circ \Psi^{ALF}_{h/2}$}, applied to $\dot z = f(z,t)$ provides second order accurate approximations of position $z$ and velocity $v=\dot z$.
%is a second order method in both position $z$ and velocity $v$.    
\end{Theorem}
The proof of the theorem is based on comparing the term\chris{s} in the Taylor series of the exact flow of a differential equation and the numerical flow obtained by composition of two steps of \chris{the} ALF method. We refer to the Appendix~\ref{sec:appendix.proof} for the computations. 
%
%\chris{[Do we need to say anything about stability?]
The composition of two steps of the ALF method, each with time-step $\frac{h}{2}$, will be \textcolor{black}{called ALF2 and denoted by $\Psi^{ALF2}_h$}. Now we are in a classical situation, with ALF2 a one step reversible method of even order and we can apply the composition methods to construct higher order methods. In this work we consider the Yoshida approach \cite{Yoshida1990}. Yoshida composition permits to  construct methods of a higher accuracy by composing numerical methods of order $2k$ for some integer $k$. It is defined by a symmetric composition of the same method $\Psi^{\mathrm{2k}}$ ($2k$ stands for the order) with different step sizes
$$ \Psi^{\mathrm{Y}}_h = \Psi^{\mathrm{2k}}_{ah}  \circ \Psi^{\mathrm{2k}}_{bh} \circ \Psi^{\mathrm{2k}}_{ah} $$
with time-steps defined by
$$ a = \frac{1}{2 - 2^{\frac{1}{2k+1}}} , \qquad b = 1-2a. $$
\begin{Theorem}[\cite{Yoshida1990}]
	Yoshida composition of a reversible method $\Psi^{\mathrm{2k}}$ of order $2k$ has order $2k+2$ and is reversible. 
\end{Theorem}
Yoshida is not the only composition method that can be used, another possible approach is  Suzuki composition \cite{Suzuki1991}. Several approaches are reviewed in \cite{hairer2006,blanes2024splittingmethodsdifferentialequations}.
\sofya{\begin{remark}
		%Note that tThe approach using Yoshida composition might require checkpoints in case of certain neural ODEes, e.g., when the learning task is to learn a heat equation. This is because the Yoshida composition forces the use of negative time-steps, which can be a problem in dissipative cases, where it can lead to instability.
		The approach based on Yoshida composition might require checkpoints in case of certain neural ODEs, e.g., when the learning task is to learn a dispersive partial differential equation such as heat equation. This is because the Yoshida composition forces the use of negative time-steps, which can be a problem in dissipative cases, where it can lead to instability. %In this case the instability in the propagation backward can lead to the exponential growth of the rounding errors, and therefore, errors in the computation of $z_N, \dots, z_1$.
\end{remark}}
%It is defined as follows
%	$$ \Phi^{\mathrm{S}}_h =  \Psi^{\mathrm{2k}}_{ah}  \circ  \Psi^{\mathrm{2k}}_{ah}  \circ   \Psi^{\mathrm{2k}}_{bh} \circ  \Psi^{\mathrm{2k}}_{ah} \circ  \Psi^{\mathrm{2k}}_{ah} $$
%	with time-steps defined by
%	$$a = \frac{1}{4 - 4^{\frac{1}{2k+1}}}, \qquad b = 1-4a. $$
%\begin{Theorem}[\cite{}]
%    Suzuki  composition of a reversible method $\Psi^{\mathrm{2k}}$ of order $2k$ has order $2k+2$ and is reversible.
%\end{Theorem}
%Suzuki composition has a higher computation costs than Yoshida due to a higher number of composition elements, but it also has a lower numerical error. 
%
%Applying now Yoshida composition to ALF2 in Example~\ref{ex:ALF2}, we obtain an order four method. If we apply the Yoshida composition again, then we obtain a 6th order. This can be seen in Figure~\ref{fig:Y4Y6}.
%
%
%\begin{figure} 
%%\vspace{-2mm}
%\centering
%\includegraphics[width=0.4\textwidth]{Y4.eps} \ \ \ 
%\includegraphics[width=0.4\textwidth]{Y6.eps}
%\fbox{\rule[-.5cm]{0cm}{4cm} \rule[-.5cm]{4cm}{0cm}}
%\caption{Order of Yoshida composition of ALF2 is on the left and the resulting method is composed again using the Yoshida approach leading to a 6th order method, which is displayed on the right.}
%\label{fig:Y4Y6}
%%\vspace{-2mm}
%\end{figure}
%
%
%In the following, we will denote by $\Psi^{Y,2k}_h$ the higher order methods obtained by Yoshida composition of ALF2, where $2k$ is the order of the method.\textcolor{black}{SO: I think this sentence is doubled}
%
%Yoshida composition permits to construct higher order reversible methods. 
In the following, we will denote by $\Psi^{Y,2k}_h$ the higher order methods obtained by Yoshida composition of ALF2, where $2k$ is the order of the method. The constructed higher order reversible method can be used for the construction of a reversible network. In this case, the step forward and the step backward are defined recursively based on the steps forward and backward of a lower order method $\Psi^{Y,2k-2}_h$. The starting method of order $2$ is the ALF2 method, i.e.\ $\Psi^{Y,2}_h=\Psi^{ALF2}_h$ by abuse of notation.

\begin{algorithm}
	\setstretch{1.4}
	\caption{Step forward of $2k$-th order Yoshida}\label{alg:step.forward}
	\begin{algorithmic}
		\State 1. Input: $(z_j,v_j,t_j,h_j)$ 
		\State 2. Set $ a = 1/(2 - 2^{\frac{1}{2k+1}}), \ b = 1-2a. $%\frac{1}{2 - 2^{\frac{1}{2k+1}}} , \qquad b = 1-2a. $
		\State 3. Set $(\tilde{z}_1,\tilde{v}_1)  = \Psi^{Y,2k-2}_{ah_j}(z_j,v_j,t_j), \ \tilde{t}_1 = t_j + ah_j$
		\State 4. Set $(\tilde{z}_2,\tilde{v}_2) = \Psi^{Y,2k-2}_{bh_j}(\tilde{z}_1,\tilde{v}_1,\tilde{t}_1), \ \tilde{t}_2 = \tilde{t}_1 + bh_j$
		
		\State 5. Set $(z_{j+1},v_{j+1})  = \Psi^{Y,2k-2}_{ah_j}(\tilde{z}_2,\tilde{v}_2,\tilde{t}_2), \ t_{j+1} = \tilde{t}_2 + ah_j$
		\If{adaptive time-stepping}
		\State 6a. compute the error of $z_{j+1},v_{j+1}$ w.r.t.\ the output of a $(2k+1)$st order integration method 
		\State 6b. compute the new $h_{j+1}$ following \cite{hairer2006}
		\Else
		\State 6. $h_{j+1} = h_j$
		\EndIf
		\State 7. Output: $(z_{j+1},v_{j+1},t_{j+1},h_{j+1})$
		%	%	\EndIf
	\end{algorithmic}
\end{algorithm}

\paragraph{Adaptive stepping} One of the main advantages in the construction of reversible methods based on ALF is that they allow for adaptive step sizes \cite{hairer2006}. This can be done in the same manner as for ALF \cite{zhuang2021mali}, where the main idea is to delete the computational graph and all the variables needed for the step size computations and only the value of the accepted new step size $h_j$ is saved. As a result, values $h_1, \dots, h_N$ are saved and then accessed in the integration backward needed for gradient computations. This can be done in exactly the same manner for our Yoshida-based methods. The resulting steps forward and backward are summarized in Algorithm~\ref{alg:step.forward} and Algorithm~\ref{alg:step.backward}.%at the end of each step forward the step size $h$ is tested using any numerical method or order $2k+1$.
%
%\begin{multicols}{2}
%\vspace{-2mm}
%
%\columnbreak

%
%\paragraph{Adaptive stepping} One of the main advantages in the construction of reversible methods based on ALF is that they allow for adaptive step sizes \cite{hairer2006}. This can be done in the same manner as for ALF \cite{zhuang2021mali}, where the main idea is to delete the computational graph and all the variables needed for the step size computations and only the value of the accepted new step size $h_j$ is saved. As a result, values $h_1, \dots, h_N$ are saved and then accessed in the integration backward needed for gradient computations. This can be done in exactly the same manner for our Yoshida-based methods. The resulting steps forward and backward are summarized in Algorithm~\ref{alg:step.forward} and Algorithm~\ref{alg:step.backward}.%at the end of each step forward the step size $h$ is tested using any numerical method or order $2k+1$.
%
%
%
\begin{remark}
	Notice that even though the reversible Heun method was proved to be of order $(2,1)$ in $(z,v)$ in \cite{Kidger2021a}, it was noted in the same paper that it gains the second order in both variables at even steps. This implies that the composition approach can be used in this case as well.%construction that we present for higher order reversible method based on ALF method can be done for the reversible Heun method as well.
\end{remark}
%
%\end{multicols}
%
%
%\chris{[There is a formatting issue here causing the next paragraph to start with "k"]}
%
%
%
%%
%
\begin{algorithm}
	\setstretch{1.55}
	\caption{Step backward of $2k$-th order Yoshida}\label{alg:step.backward}
	\begin{algorithmic}
		\State 1. Input: $(z_{j+1},v_{j+1},t_{j+1},h_{j+1})$
		\State 2. Set $ a = 1/(2 - 2^{\frac{1}{2k+1}}), \ b = 1-2a. $%\frac{1}{2 - 2^{\frac{1}{2k+1}}} , \qquad b = 1-2a. $
		\State 3. Set $(\tilde{z}_1,\tilde{v}_1)  = \Psi^{Y,2k-2}_{-ah_{j+1}}(z_{j+1},v_{j+1},t_{j+1}), \ \tilde{t}_1 = t_{j+1} - ah_{j+1}$
		\State 4. Set $(\tilde{z}_2,\tilde{v}_2) = \Psi^{Y,2k-2}_{-bh_{j+1}}(\tilde{z}_1,\tilde{v}_1,\tilde{t}_1), \ \tilde{t}_2 = \tilde{t}_1 - bh_{j+1}$
		
		\State 5. Set $(z_{j},v_{j})  = \Psi^{Y,2k-2}_{-ah_{j+1}}(\tilde{z}_2,\tilde{v}_2,\tilde{t}_2), \ t_{j} = \tilde{t}_2 - ah_{j+1}$
		
		\State 6. Set $h_j$ from $h_1, \dots, h_N$ obtained in the integration forward
		\State 7. Output: $(z_{j},v_{j},t_{j},h_j)$
		%	%	\EndIf
	\end{algorithmic}
\end{algorithm}
%\vspace{-3mm}
\paragraph{Gradient computations} The augmentation of the feature space leads to the new variable which we denote \chris{by} $\phi = (z,v)$. Then, the learning problem is formulated as follows with $P_z(\phi)$ projection of $\phi$ to $z$
\begin{equation} \label{eq:DL:optimal.problem.Y}
	\begin{aligned}
		& \min_{\{ \theta_j \}} J =  L(P_z(\phi_N), y) \\
		&  \phi_{j+1}= \Psi^{Y,2k}_h (\phi_j, \theta_j), \quad j = 0, \dots, N-1, \\
		&\phi_0 = (z,f(z,\theta_0)).
	\end{aligned}
\end{equation}
%
%The adjoint method leads to explicit formulas for the gradient computation using the step backward of the considered numerical methods as follows. For the method to provide exact gradients, it is important to obtain the correct discrete adjoint equation \eqref{eq:adjiont} corresponding to \eqref{eq:DL:optimal.problem}. 
Following \cite{Griesse2004}, the discrete version of \eqref{eq:adjiont} associated with \eqref{eq:DL:optimal.problem.Y} is given by 
\begin{equation} \label{eq:adjoint.discrete}
	\left(  \lambda_N  \right)^\top = \nabla L(\phi_N), \quad \lambda_j = \left( \frac{\partial \phi_{j+1}}{\partial \phi_j} \right)^\top \lambda_{j+1},
\end{equation}
and the gradients are computed by
\begin{equation} \label{eq:adjoint.grad.discrete}
	\frac{\partial J(\theta)}{\partial \theta_{j}} = \lambda_{j+1}^\top \frac{\partial \phi_{j+1}}{\partial \theta_{j}}.
\end{equation}
The adjoint method for the gradient computation as in \chris{the} MALI network \cite{zhuang2021mali} and \chris{the} reversible Heun \chris{network} \cite{Kidger2021a} is based on the propagation $\phi_j \rightarrow \phi_{j+1}$ and automatic differentiation  for the computation of the step backward of the adjoint variable following \eqref{eq:adjoint.discrete}. The resulting method of gradient computation is summarized in Algorithm~\ref{alg:reconstruction}. \textcolor{black}{Alternatively, the exact expression of the numerical method governing the adjoint dynamics \eqref{eq:adjoint.discrete} can be obtained, see details in Appendix~\ref{sec.adjoint}. In this case, there is no need to compute $\frac{\partial \phi_{j+1}}{\partial \phi_j}$, which makes the approach} \textcolor{black}{computationally more efficient and memory efficient.}
\begin{algorithm}
	\caption{Computation of gradients}\label{alg:reconstruction}
	\begin{algorithmic}
		\State 1. Input: training data $z_0$, initialization of parameters $\theta$, velocity $v_0 = f(z_0,\theta_0)$
		\State 2. Propagate through the network using $\Psi^{Y,2k}_h$ to get $(z_N, v_N)$
		\State 3. Set $\lambda^z_N = \nabla L(z_N, y)$ and $\lambda^v_N = 0$
		\For{\texttt{j = N-1 to 1}}
		\State 4. Compute $\phi_j$ from $\phi_{j+1}$ using Algorithm~\ref{alg:step.backward}
		\State 5. Compute $\phi_{j+1}$ from $\phi_{j}$ using Algorithm~\ref{alg:step.forward} to get the computational graph
		\State 6. Compute $\lambda_j$ from $\lambda_{j+1}$ using \textcolor{black}{\eqref{eq:adjoint.discrete}} and AD to compute $\frac{\partial \phi_{j+1}}{\partial \phi_j}$%Theorem~\ref{th:composition.grad}.
		\State 7. Compute $\frac{\partial J(\theta)}{\partial \theta_{j}}$ using \textcolor{black}{\eqref{eq:adjoint.grad.discrete}}
		\State 8. Delete $\lambda_{j+1}, \phi_{j+1}$ and the computational graphs 
		\EndFor
		
		\State 9. Output: gradients  $\frac{\partial J(\theta)}{\partial \theta_{j}}$ for $j = 1, \dots, N-1$. 
		%	\EndIf
	\end{algorithmic}
\end{algorithm}
%\vspace{-3mm}
%
\textcolor{black}{\paragraph{Costs comparison}  We will use the following notations: $d$ is the dimension of $z$, $T$ is the length of the time interval in the continuous-depth setting, $N$ is the number of layers, $M$ stands for the number of layers in $f$, when $f$ is given by a neural network itself, $s$ denotes the number of steps needed for the computation of a time-step in the adaptive step size selection, $p$ is the order of the considered numerical method and $r$ is the number of evaluations of $f$ used in the numerical method (e.g. stages in Runge-Kutta methods or compositions in our approach). We show the comparison of the new proposed approach with the standard backpropagation approach, adjoint method version NODE \cite{Chen2018}, ACA approach \cite{Zhuang2020} and MALI approach \cite{zhuang2021mali} in Table~\ref{tab:costs}, which extends the Table~1 in \cite{zhuang2021mali}. We use big $\mathcal{O}$ notation, when the constants depend on the learning tasks. }
%
%\vspace{-3mm}
\textcolor{black}{\paragraph{Computational costs} The compositional structure of the proposed method directly implies that the computation costs for gradient computations are equal to the computational costs by ALF multiplied by $r$, the number of the compositions.  Notice that $N$ depends on the order $p$ of the discretization method and becomes smaller when the order is higher for fixed $\varepsilon$ and $T$. As a result, ALF method needs more time-steps, than higher order methods for $\varepsilon < 1$, which is related to the bias in the learned parameters in the task of identification of the parameters, as explained in Appendix~\ref{sec:Kepler} and illustrated in Figure~\ref{fig:Error.landscape}, and to the training error.}
%
%\vspace{-3mm}
\textcolor{black}{\paragraph{Memory costs} The gradient computation requires to compute $\frac{\partial \phi_{j+1}}{\partial \phi_j}$ leading to the storage of all the intermediate states involved in the step forward. This increases the memory costs of MALI by a factor $r$, see Table~\ref{tab:costs}. Notice that the approach presented in Appendix~\ref{sec.adjoint} does not require to store all the intermediate states. Indeed, a step backward of the state-adjoint system is a composition of rescaled steps backward of the ALF method. Therefore, we only need to store one intermediate state obtained in the composition at a time. This makes the method of the same memory cost as MALI. Notice that depending on the depth of the network, adaptive checkpointing as 
	in \cite{Matsubara2021} can be added. When no checkpoints are needed the behaviour is as in NODE and in the worst case the behaviour is as in the backprop. In general, the number of checkpoints depends on $N$, which depends linearly on $T$. Therefore, more checkpoints are needed in case of large $T$.}
\begin{center}
	\begin{table}[]
		\caption{Comparison of costs in gradient computations for different approaches}
		\label{tab:costs}
		\centering
		\begin{tabular}{ |p{1.5cm}||p{4cm}|p{2.5cm}|p{3.5cm}| }
			\hline
			Method& Computational costs & Memory costs& Number of epochs $N$ in function of accuracy $\varepsilon^{*}$\\
			\hline
			Backprop &  $r\times d\times M \times N\times s \times 2$   & $r \times d\times M \times N\times s$& $T\times \mathcal{O}(\varepsilon^{-\frac{1}{p+1}})$\\
			NODE  &  $r\times d\times M \times N\times s \times 2$   & $d\times M $&$T\times \mathcal{O}(\varepsilon^{-\frac{1}{p+1}})$\\
			ACA & $r \times d\times M \times N\times (s+1)$  & $d\times (M + N)$&$T\times \mathcal{O}(\varepsilon^{-\frac{1}{p+1}})$\\
			MALI   & $ d\times M \times N\times (s+2)$ & $d\times (M + 1)$&$T\times \mathcal{O}(\varepsilon^{-\frac{1}{3}})$\\
			Proposed method &   $r\times d\times M \times N\times (s+2)$  & $r^{**}\times d\times (M + 1)$&$T\times \mathcal{O}(\varepsilon^{-\frac{1}{p+1}})$\\
			\hline
		\end{tabular} \\
		\footnotesize{$^{*}$ $\varepsilon$ is the error tolerance for the estimation of the local error in the stepsize selection}\\
		\footnotesize{$^{**}$ Memory costs of the proposed method can be reduced to $r = 1$, if the gradients are computed as in Appendix\ref{sec.adjoint}}
	\end{table}
\end{center}
%%
%\subsection{Learning dynamical systems}
%We consider the problem of learning dynamical systems, where the %observations of trajectories 
%
%\vspace{-10mm}
\section{Experiments}
%\vspace{-3mm}
%
\subsection{Parameter identification in dynamical systems}
We consider \chris{the identification} problem of unknown parameters of a dynamical system. The structure of the differential equations is assumed to be known, but some parameters in the equations are unknown. The training data is given by \textcolor{black}{snapshots of} trajectories $\{x_l(t_i)\}_{i,l}$ with $l= 1, \dots L, \ i = 0, \dots, I$. The goal is to learn the parameters from the given trajectories. This class of problems can be naturally treated using the neural ODE approach. The vector field $f$ in \eqref{eq:ODE} is given by the known differential equation and $\theta = \theta_1, \dots, \theta_s$ is the set of unknown parameters. In this case, the learning problem can be stated in the form \chris{of} \eqref{eq:DL:optimal.problem.Y}, where the same $\theta_1, \dots, \theta_s$ appear all at each layer. The training data $z_0 = (x_1(t_0), \dots, x_L(t_0))$ stands for the initial points and $y$ includes all the other points of the given trajectories. We denote by $y(t_i)$ the points in $y$ corresponding to trajectories at time $t_i$ for $i = 1, \dots, I$.  %stands for the final points of the training trajectories issued from~$z_0$. 
The loss have a particular structure in this case as it depends on the intermediate states obtained during the integration of neural ODE, namely, it depends on $(z_{N_1}, \dots, z_{N_I})$, to measure the distance with the given trajectories points $(y(t_1), \dots, y(t_I))$. As a result, it takes the form  $L = \sum_{i=1}^I L_i(z_{N_i}, y(t_i))$. Because of the additive form of the loss, the gradients can be computed as a sum of the corresponding gradients of $L_1, \dots, L_I$ as follows
%\vspace{-1mm}
$$\frac{\partial L}{ \partial \theta_i} =  \frac{\partial L_1}{ \partial \theta_i} + \cdots + \frac{\partial L_I}{ \partial \theta_i}, \qquad i=1, \dots, s,
%\vspace{-1mm}
$$
where each of the terms in the sum is computed using Algorithm~\ref{alg:reconstruction}. The memory efficiency is still important in this case, because we do not store all the intermediate states at the propagation forward, but only the states which approximate the trajectories at the desired times $t_1, \dots, t_I$. %In our numerical experiments we set $L_i(z_{N_i}, y(t_i)) = \frac{1}{L}\|z_{N_i} - y(t_i)\|_{l_2}^2$ with $L$ the number of trajectories and compare the performance of numerical integrators $\Psi^{Y,4}$ and $\Psi^{ALF}$ in the training.
%%
%
%measures the difference between the trajectories obtained using numerical integration by $\Psi^{Y,4}$ or $\Psi^{ALF}$  with the parameter values $\theta_1, \dots, \theta_s$ and the training trajectories, namely $L = \|z_N(\theta) - y\|_{l_2}^2$ \chris{with $z_N(\theta)$ the endpoints of all the} \chris{numerical} trajectories parameterized by $\theta = \theta_1, \dots, \theta_s$. % and \chris{with the} $i$th trajectory from the training set $x_i$.
%
%\chris{[Is it true that only errors at the endpoints of trajectories are used to compute the loss? Are the two $L$s the same?]}
%\textcolor{black}{SO: is this a comparison at the final points of the trajectories or along the trajectory? It would be clearer if you also explain what $t_i$ is.}
%
%\vspace{-3mm}
\textcolor{black}{
	\paragraph{Statistical inference} In simulation based inference or likelihood-free inference probabilistic methods are employed to identify parameters in models based on repeated forward simulations \cite{Cranmer2020}. Traditionally, these consider the forward pass as a black box (such as Approximate Bayesian Computation (ABC) \cite{Rubin1984,Beaumont2002}) and do not require differentiability with respect to the model parameters or the inputs. This needs to be contrasted to our proposed neural network architecture, which is designed to circumvent large memory requirements in the computation of gradients when the layers are wide. Indeed, a combination of our architecture with simulation based inference models that do make use of gradients such as \cite{Graham2017} constitutes an interesting avenue for future research.}
%\vspace{-2mm}
%
\subsubsection{Kepler problem}
%\vspace{-2mm}
We consider the Kepler problem, where the dynamics describes the evolution of the position $q$ and velocity $v$ of a mass point moving around a much heavier body. \chris{It is modeled on the 4}-dimensional space $x = (q,v) \in \R^2 \times \R^2$.
%We assume that one parameter in the equations is unknown. 
The equations are defined on \textcolor{black}{the} time interval $[0,1]$ as follows
%\vspace{-1mm}
\begin{equation} \label{eq:Kepler}%\tag{$**$}
	%\begin{aligned}
	\dot{q} = v, \quad \dot{v} = - \frac{\alpha}{\| q\|^{3}} \, q, 
	%\end{aligned}
	%\vspace{-1mm}
\end{equation}
with an unknown parameter $\alpha \in \R$. The training set is given by the initial condition $x(t_0)$ and $q$-coordinate of $5$ points on a trajectory of \eqref{eq:Kepler} generated with $\alpha = \pi/4 \approx 0.785$, i.e.\ $\{ q(t_i)\}^5_{i=1}$. The task is to learn $\alpha$ as accurately as possible. From the training set we form $z_0 = x(t_0)$ and the corresponding $y(t_i) =  q(t_i)$ for $i = 1, \dots, 5$. %To learn $\alpha$ 
We set up a learning problem in \chris{the} form \chris{of} \eqref{eq:DL:optimal.problem.Y} with the loss defined by $L = \sum_{i=1}^5\|q_{N_i} - q(t_i)\|^2$ with $q_{N_i}$ projection of $z_{N_i}$ to $q$-coordinate and $N_i$ the number of time-steps used in the integration from $t_{i-1}$ to $t_i$. We compare two algorithms \chris{for performing the numerical integration during training}, namely, ALF and the Yoshida composition of ALF2 of order 4. We write Y4 for the Yoshida composition method for shortness. %\begin{itemize}
%\item Neural ODE given by the Kepler dynamics 
%	\item Integrate \eqref{eq:Kepler} with ALF or Yoshida composition of ALF2 to get $x_i$
%	\item learn parameter $\alpha$ by minimizing the loss $L = \sum_{i}  \|x_i - x(t_i)\|_{l_2}^2$
%\end{itemize} 
We test the \chris{wall-clock} time required to reach loss accuracy $10^{-8}$ using adaptive methods. The tests are run for different initializations of $\alpha$ in optimization. %, while the true value $\alpha = \frac{\pi}{4} \approx 0.785$ is used to generate training data. 
The results can be seen in Table~\ref{table:Kepler_time}. It can be observed that in all tests, Y4 is at least four times faster than ALF. 
\begin{table}
	\caption{Time to reach accuracy $10^{-8}$ using adaptive methods in the Kepler problem}
	\label{table:Kepler_time}
	\centering
	\begin{tabular}{ |p{3cm}||p{3cm}|p{3cm}| }
		\hline
		\multicolumn{3}{|c|}{Computation time} \\
		\hline
		Initial value of $\alpha$ & adaptive ALF  &  adaptive Y4\\
		\hline
		0.1 &  7.68 sec  &  {2.42 sec}   \\
		\hline
		%	0.6   &  4.76 sec   & {1.89 sec}  \\
		%	\hline
		0.7   &  4.07 sec   &  {1.02 sec}  \\
		\hline
		0.75   &  3.26 sec    &  {0.803 sec}    \\
		\hline
		0.8   &  2.5 sec   & { 0.44 sec}  \\
		%	\hline
		%	0.9&    5.65 sec   & { 1.25 sec } \\
		\hline
		1.3 & 8.39 sec & { 3.85 sec }\\
		\hline
	\end{tabular}
\end{table}
\textcolor{black}{The reason of the faster training for Y4 is in using larger step sizes for the forward integration. The lower order method requires smaller step sizes to reach the same accuracy \textcolor{black}{defined by an error tolerance} and this leads to more steps in the computation of trajectories. Additional results for the Kepler problem supporting the reasoning can be found in Appendix~\ref{sec:Kepler}.}
%Next, we compare the accuracy in learning the parameter using %fixed step methods.
%methods with fixed step-sizes.
%The accuracy \chris{is} analyzed using the error landscape for ALF and Y4 for fixed step size $h = 0.01$. The results in Figure~\ref{fig:Error.landscape} show that the minimizer of Y4 method is closer to the true value of the parameter than the minimizer of ALF, which means that Y4 will provide more accurate results in the optimization. % independently of the algorithm used for the optimization. %in the learning.
%\textcolor{black}{SO: Which solver used in learning is meant here?}
%
%\vspace{-2mm}
\subsubsection{Nonlinear harmonic oscillator} \label{sec:nonlin.osc.param}
%\vspace{-2mm}
In the second example we consider a system of coupled Duffing oscillators, which describes the movements of a coupled system of mass points attached with springs with nonlinear elastic forces. The dynamics of $N$ mass points is given by the following equations
%\vspace{-2mm}
\begin{equation} \label{eq:coupled_osc}%\tag{$**$}
	%\begin{aligned}
	\dot{q}_i = v_i, \quad 
	\dot{v}_i = - a_iq_i - b_iq_i^3 - \sum_{j=1}^{N} e_{i,j}(q_i - q_j),  \qquad i = 1, \dots, N,\\
	%&\dot{q}_2 = v_2, \quad \dot{v}_2 = - a_2q_2 - b_2q_2^3 - e(q_2 - %q_1),  
	%\end{aligned}
	%\vspace{-2mm}
\end{equation}
with the condition $e_{i,j} = e_{j,i}$. Positions of $N$ mass points are given by $q = (q_1, \cdots, q_N) \in \R^N$ and velocities by $v = (v_1, \cdots, v_N) \in \R^N$. We set $x =(q,v) \in \R^{2N}$. In the numerical experiments we fix $N = 10$ and assume that parameters $a_i, b_i, e_{i,j} \in \R$ for $i,j = 1, \dots, 10$ are unknown. As a result, \eqref{eq:coupled_osc} has dimension $20$ with $65$ unknown parameters. The training set consists of initial and final positions of 200 trajectories, that is $z_0 = (x_1(t_0), \dots, x_{200}(t_0))$ and $y = y(t_1) = (x_1(t_1), \dots, x_{200}(t_1))$. In this setting, we compare  the computational time to reach a certain training accuracy of ALF and Y4 with adaptive time-stepping and the training accuracy of ALF and Y4 with fixed step-size. %In all the tests the true values of the parameters are set to $a_1 = 2, \ a_2 = 0.7, \ b_1 = -0.4, \ b_2 = 3, \ e = 1$. 
We present the \chris{wall-clock times} to reach the training accuracy $10^{-4}$ %results on the learning using the adaptive methods
in Table~\ref{table:Osc1}. The time required by Y4 to reach accuracy $10^{-4}$ is almost two times smaller which
%proves
illustrates the lower computational costs of the method. \textcolor{black}{As before, the ALF method with adaptive time-stepping requires smaller step sizes and more steps are used in each epoch of the optimization. Details with additional results confirming the behaviour are presented in Appendix~\ref{sec:annex.nonlin.osc}.} % In Table~\ref{table:Osc2} we show the results in accuracy of the learned parameters in case $h=0.01$. 
%, when the learning is done until accuracy $10^{-5}$ in the loss. We can observe that the accuracy of Y4 is one order higher than the one by ALF.
%
%\chris{[It is not clear to me what `accuracy' refers to in the tables and the text proceeding them. Is it the error of the parameters or the value of the loss function?]}
%
% 
\begin{table}
	%\vspace{-2mm}
	\caption{Time to reach accuracy $10^{-4}$ using adaptive methods in nonlinear oscillators problem}
	\label{table:Osc1}
	\centering
	\begin{tabular}{ |p{6cm}||p{3cm}|p{3cm}| }
		\hline
		\multicolumn{3}{|c|}{Computation time} \\
		\hline
		Mean parameter error at initialization & adaptive ALF  &  adaptive Y4\\
		\hline
		$0.28897009$
		&  $81608$ sec  &  {$51720$ sec}   \\
		\hline
		%	0.6   &  4.76 sec   & {1.89 sec}  \\
		%	\hline
		$0.29821727$  &   $68645$ sec   &  {$40646$ sec}  \\
		\hline
		$0.30549358 $   &  56764 sec    &  {29990 sec}    \\
		\hline
		$0.30289593$   &  96301 sec   & { 46524 sec}  \\
		%	\hline
		%	0.9&    5.65 sec   & { 1.25 sec } \\
		\hline
		$0.29106813$  & 22161 sec & { 13790 sec }\\
		\hline
	\end{tabular}
	%\vspace{-2mm}
\end{table}
% \begin{table}
	% %\vspace{-2mm}
	%   \caption{Accuracy in the learned parameters in nonlinear oscillators problem}
	%   \label{table:Osc2}
	%   \centering
	% \begin{tabular}{ |p{6cm}||p{3cm}|p{3cm}| }
		% 	\hline
		% 	\multicolumn{3}{|c|}{Error in parameter accuracy in learning with fixed step size $h = 0.1$} \\
		% 	\hline
		% 	Mean parameter error at initialization &  ALF  &   Y4\\
		% 	\hline
		% 	$0.28897009$ %$a_1 \approx 1.629$, $a_2 \approx 0.751$, $b_1 \approx -0.385$, $b_2 \approx 3.513$, $e \approx 1.167$ 
		%  &  $2.053 10^{-02}$  &  $3.396 10^{-03} $  \\
		% 	\hline
		% %	0.6   &  4.76 sec   & {1.89 sec}  \\
		% %	\hline
		% 	$0.28897009$ %$a_1 \approx 1.9$, $a_2 \approx 0.792$, $b_1 \approx -0.28$, $b_2 \approx 2.644$, $e \approx 0.788$  
		%  &  $2.019 10^{-02}$   &  $6.285 10^{-03}$ \\
		% 	\hline
		% 	$0.28897009$ %$a_1 \approx 1.923$, $a_2 \approx 0.5$, $b_1 \approx -0.411$, $b_2 \approx 2.883$, $e \approx 0.952$  
		%  &  $2.011 10^{-02}$    &  $1.865 10^{-03}$  \\
		% 		\hline
		% 	$0.28897009$ %$a_1 \approx 1.666$, $a_2 \approx 0.855$, $b_1 \approx -0.329$, $b_2 \approx 3.753$, $e \approx 0.993$ 
		%  &  $2.002 10^{-02}$   & $8.815 10^{-03}$ \\
		% %	\hline
		% %	0.9&    5.65 sec   & { 1.25 sec } \\
		% 	\hline
		% 	$0.28897009$ %$a_1 \approx 2.471$, $a_2 \approx 0.629$, $b_1 \approx -0.477$, $b_2 \approx 2.175$, $e \approx 0.7645$  
		%  & $2.003 10^{-02}$
		%  & $8.567 10^{-03}$ \\
		% 	\hline
		% \end{tabular}
	% %\vspace{-2mm}
	% \end{table}
%%
%
%
%\vspace{-2mm}  
\subsection{Learning of dynamical systems parameterized by neural network}
%\vspace{-2mm}
We consider a problem, where a part of the structure of the differential equations is known and the unknown part is approximated using a neural network. Our goal is to find the neural network parameterization such that the resulting trajectories of the system are as close as possible to given trajectories from the training set. As before, the problem can be treated by the neural ODE approach. In this case the vector field $f$ in \eqref{eq:ODE} is given by a neural network. %The learning problem \eqref{eq:DL:optimal.problem.Y} is then to find the parameters of the network. As before, we need to integrate $f$ by a reversible method. The training is done as in the problem of identification of parameters. The only difference is that the true vector field is only approximated by a network and we learn the best possible approximation. %%%The training data $z_0$ is the set of initial points of trajectories and  $y$ stand for the final points on the training trajectories. Loss  $L(\Phi(\phi_N),y)$ measures the difference between the trajectories obtained using the numerical integration by $\Psi^{Y,2k}_h$ and the value of parameters $\theta_1, \dots, \theta_s$ and the training trajectories, namely $L = \sum_{i}  \|x_i - x(t_i)\|_{l_2}^2$. 
%\vspace{-2mm}
\subsubsection{Nonlinear harmonic oscillator} \label{sec:nonlin.osc.nn}
%\vspace{-2mm}
We consider the problem of approximating the potential function of a physical system, %In this case, the vector field of the physical system is approximated by a neural network and then integrated to obtain the final time points of the trajectories. The result is compared with the true trajectories. We consider again 
\textcolor{black}{given by the Duffing oscillators with two mass points. Equations can be equivalently written as}
%\vspace{-1mm}
\begin{equation} \label{eq:coupled_osc.potential}%\tag{$**$}
	%\begin{aligned}
	\dot{q} = v, \quad \dot{v} = -\nabla V(q), 
	%\vspace{-1mm}
	%\end{aligned}
\end{equation}
with $q = (q_1, q_2), \ v = (v_1, v_2)$ and $V(q)$ stands for the potential energy of the system. %given by 
% \begin{equation} \label{eq:potential.nonlin.osc}
	%  V(q) = \frac{a_1}{2}q_1^2 + \frac{a_2}{2}q_2^2 + \frac{e}{2}(q_2- q_1)^2 + \frac{b_1}{4}q_1^4 + \frac{b_2}{4}q_2^4.   
	% \end{equation}    
The learning task is to learn $V(q)$. %the unknown potential. 
\sofya{The gradient of the potential is approximated using a neural network with 51500 parameters.} \textcolor{black}{To obtain the potential from the learned vector field, we apply numerical integration methods to the neural network approximating $-\nabla V(q)$.} % which approximates the unknown vector field governing $\dot v$. 
We compare the computational time of ALF and Y4 to \chris{reduce the value of the loss function below} $10^{-2}$. % in the loss.
The results presented in Table~\ref{table:OscNN} show \textcolor{black}{that Y4 is faster than ALF in completing the training on different random initializations of the network parameters.}%the better performance of the higher order method \textcolor{black}{Formulate better the observed advantages (what is better and what are the conclusions), refer to Appendix for more numerical results!!!!!}. %This suggests that higher order reversible methods can be used in such learning tasks \chris{to accelerate training.}
%for acceleration of the learning. 
%
\begin{table}
	%\vspace{-2mm}
	\caption{Time to get accuracy $10^{-2}$ by adaptive methods in oscillators problem~parameterized~by~NN}
	\label{table:OscNN}
	\centering
	\begin{tabular}{ |p{6cm}||p{3cm}|p{3cm}| }
		\hline
		\multicolumn{3}{|c|}{Computation time} \\
		\hline
		Random initialization of parameters in NN & adaptive ALF  &  adaptive Y4\\
		\hline
		Initialization 1
		&  2504 sec  &  {1974 sec}   \\
		\hline
		%	0.6   &  4.76 sec   & {1.89 sec}  \\
		%	\hline
		Initialization 2  &   2961 sec   &  {1857 sec}  \\
		\hline
		Initialization 3   &  4185 sec    &  {2627 sec}    \\
		\hline
		Initialization 4   &  3542 sec  & {2125 sec}  \\
		%	\hline
		%	0.9&    5.65 sec   & { 1.25 sec } \\
		\hline
		Initialization 5  & 3396 sec & { 2616 sec }\\
		\hline
	\end{tabular}
	%\vspace{-2mm}
\end{table}
\textcolor{black}{
	%\vspace{-2mm}
	\subsubsection{Discretized wave equation} \label{sec:wave.nn}
	%\vspace{-2mm}
	In the second example we consider \textcolor{black}{the} 1-dimensional wave equation $u_{tt}(t,x) = u_{xx}(t,x) - \nabla V(u(t,x))$ \textcolor{black}{on the spatial-temporal domain $[0, 1] \times [0,0.3]$} %{\color{red} is it possible to write here $\mathcal{D}=[0,T]\times [a,b]$ and precise the values of T,a,b only in the table 5? to avoid the question like: why this specific domain?} 
	with periodic boundary conditions \textcolor{black}{in space for the potential $V(u)=\frac 12 u^2$. On a spatial, equidistant, periodic mesh with mesh width $\Delta x = \frac 1 {40}$ we seek to describe the system's evolution by the first order system }
	%and given initial condition $u(0,x) = u_0(x)$, where $u:  [0, 0.5] \times [0,0.1] \rightarrow \R$ is a wave function and $V: L_2 \rightarrow \R$ is a wave potential. We apply 5 point stencil to discretize the wave equation, resulting in the following second order ODE
	%\vspace{-1mm}
	\begin{equation} \label{eq:disc:PDE}
		\dot u_d = v_d, \quad \dot v_d = f(u_d), 
		%\vspace{-1mm}
	\end{equation}
	\textcolor{black}{where the unknown function $f$ is parametrized as a fully connected ReLU neural network with one hidden layer of size $100$.}
	%with $f(q)$ function, which is unknown and approximated by a neural network.
	The dimension of $(u_d, v_d)$ is 40.
	We compare \textcolor{black}{the training performance of} ALF and Y4.
	\textcolor{black}{Both adaptive methods are employed with the same error tolerance.}
	Yoshida is faster in finishing each epoch and \textcolor{black}{the optimizer takes less time to minimize the training loss below} $10^{-3}$. The precise results are \textcolor{black}{reported} in Table~\ref{table:PDENN} for 5 \textcolor{black}{random} initializations in the training.}
\textcolor{black}{This illustrates the applicability of our method to the highly active research area of learning models of systems that are governed by partial differential equations.}
\begin{table}
	%\vspace{-2mm}
	\caption{\textcolor{black}{Wallclock} time to get \textcolor{black}{the training loss below} $10^{-3}$ by adaptive methods in discretized~PDE}
	\label{table:PDENN}
	\centering
	\begin{tabular}{ |p{6cm}||p{3cm}|p{3cm}| }
		\hline
		\multicolumn{3}{|c|}{Computation time} \\
		\hline
		Random initialization of parameters in NN & adaptive ALF  &  adaptive Y4\\
		\hline
		Initialization 1
		&  336.6808 sec  &  {138.5936 sec}   \\
		\hline
		%	0.6   &  4.76 sec   & {1.89 sec}  \\
		%	\hline
		Initialization 2  &   169.5010 sec   &  {133.6825 sec}  \\
		\hline
		Initialization 3   &  180.8185 sec    &  {140.3172 sec}    \\
		\hline
		Initialization 4   &  153.7389 sec  & {128.2795 sec}  \\
		%	\hline
		%	0.9&    5.65 sec   & { 1.25 sec } \\
		\hline
		Initialization 5  & 142.7732 sec & { 142.9196 sec }\\
		\hline
	\end{tabular}
	%\vspace{-2mm}
\end{table}
%
%\vspace{-3mm}
\section{Conclusion}
%\vspace{-3mm}
In this work, we \chris{construct} %proposed a construction approach for 
higher order reversible methods. \chris{These constitute explicit numerical integrators which are compatible with adaptive step-size selection strategies. The methods are employed to train deep neural networks that are based on neural ODEs. Thanks to the reversibility property, we avoid high memory requirements for backpropagation in the optimization procedure of the network parameters.}
\textcolor{black}{Memory efficient backpropagation allows an application of deep architectures to the identification tasks of models of high-dimensional dynamical systems, which arise, for instance, as spatial discretizations of partial differential equations. As the method is based on neural ODEs, it can be trained with time-series data at irregular time-steps and can predict continuous time-series data.}
%and the corresponding network architecture.
We showed the advantages of the newly constructed networks on the example of a network based on a 4th order method
\textcolor{black}{and demonstrate lower memory costs and 
	faster training in comparison to lower order methods. }
%The advantages are twofold. On the one hand, the higher order method leads to lower computational costs which
%is expressed in smaller
%\chris{results in faster training.}
%computational time in case of adaptive stepping.
%\chris{Moreover, in parameter identification tasks, higher }%On the other hand, better
%accuracy is reached using higher order method \chris{at a given} %fixed
%step size.

\textcolor{black}{While the examples in the article focus on system identification tasks for systems governed by differential equations, extensions to neural stochastic differential equations \cite{Kidger2021a} are of interest and applications to normalizing flows or image processing \cite{Blanchette2020} can be an exciting avenue to explore in future works.
}

%Such a construction of a neural network can be used for approximation properties of general functions. We only identified the advantages on the tasks of learning of dynamical systems. The advantages of the suggested network in case of other learning tasks has to be further investigated. One %straightforward applications
%\chris{direction for future research are higher order reversible techniques to train} neural SDEs as investigated in \cite{Kidger2021a}.

%It would be of interest to analyse the resulting order of approximation in this case. 

%\subsubsection*{Acknowledgments}
%Use unnumbered third level headings for the acknowledgments. All
%acknowledgments, including those to funding agencies, go at the end of the paper.

\bibliography{Sofya_biblio}
\bibliographystyle{plain} 

\appendix
%\section{Appendix}
\section{Error analysis of ALF2} \label{sec:ALF2}
\textcolor{black}{\begin{example} \label{ex:ALF2} Consider a simple example of a differential equation on $\R$ given by
		\begin{equation} \label{eq:example}
			\dot z = z^2 + t + \sin(zt) + \frac{1}{z^2 + 1}.
		\end{equation}
		We solve the equation numerically using \textcolor{black}{the} ALF method and compare with a solution of high accuracy for different step sizes. The results are plotted in Figure~\ref{fig:ALF} and show the second order behaviour in both $(z,v)$ variables.
\end{example}}
\begin{figure}[h] 
	\centering
	\includegraphics[width=0.4\textwidth]{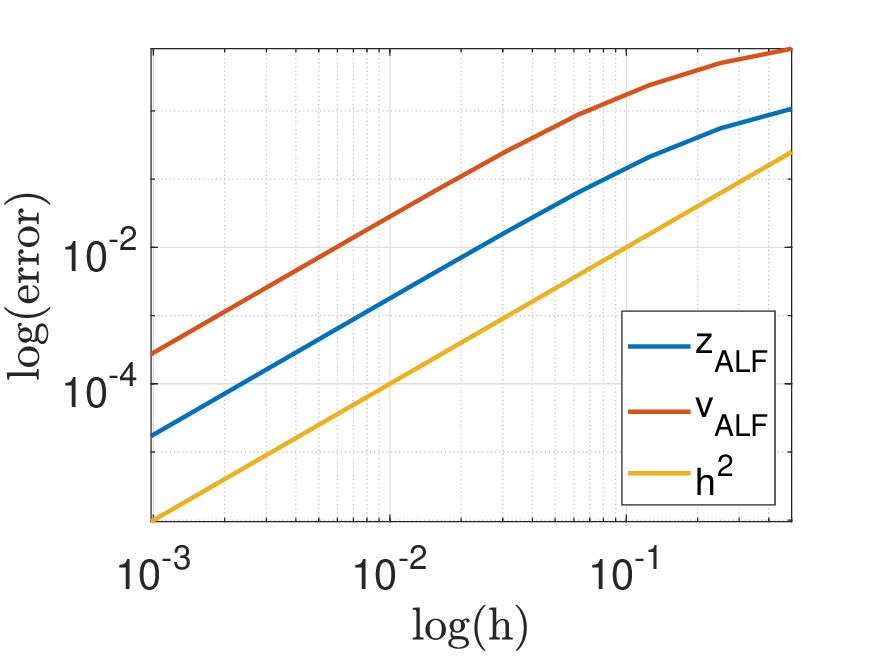}
	%\fbox{\rule[-.5cm]{0cm}{4cm} \rule[-.5cm]{4cm}{0cm}}
	\caption{Log-log plot of the global error of trajectories $(z(t),v(t))$ of \eqref{eq:example} defined on time interval $[0, 1.0]$ and obtained by ALF with $h$ ranging from $0.5$ to $10^{-3}$.}%\textcolor{black}{SO: longer caption such that figure is self explained without reading the text}}
\label{fig:ALF}
\end{figure}
\subsection{Proof of Theorem~\ref{th:ALF2}} 
\label{sec:appendix.proof}
%In this appendix we show the proof of Theorem~\ref{th:ALF2}. Namely, 
We show that the local error of ALF2 in $(z,v)$ is of order $O(h^3)$. Let us consider the Taylor expansion of the exact flow $(z(t), v(t))$ around $(z(t_0), v(t_0) = f(z_0,t_0))$.
$$z(t_0 + h) = z_0 + h f(z_0, t_0) + \frac{h^2}{2}\left( \frac{\partial f}{\partial z}(z_0, t_0) \circ f(z_0, t_0) + \frac{\partial f}{\partial t}(z_0, t_0) \right) + O(h^3),$$
\begin{multline*}
v(t_0 + h) =  f(z_0, t_0) + h \left( \frac{\partial f}{\partial z}(z_0, t_0) \circ f(z_0, t_0) + \frac{\partial f}{\partial t}(z_0, t_0)  \right) + \\
\frac{h^2}{2}( \frac{\partial^2 f}{\partial z^2}(z_0,t_0)(f(z_0,t_0), f(z_0,t_0)) +  \frac{\partial f}{\partial z}(z_0, t_0) \circ \frac{\partial f}{\partial z}(z_0, t_0) \circ f(z_0, t_0) + 2 \frac{\partial^2 f}{\partial z \partial t}(z_0,t_0)\circ f(z_0,t_0)  \\
+ \frac{\partial f}{\partial z}(z_0, t_0) \circ \frac{\partial f}{\partial t}(z_0, t_0) + \frac{\partial^2 f}{\partial t^2}(z_0,t_0) )  + O(h^3).
\end{multline*}
Now we consider the same for the numerical flow obtained with ALF2, that is composition of two steps of ALF each with the step size $\frac{h}{2}$. One step of ALF2 from $(z_0, v_0)$ leads to $(z_1, v_1)$ of the form
\begin{multline*}
z_1(h)  = z_0 + \frac{h}{2} ( f(z_0 + \frac{h}{4}f(z_0,t_0), t_0+ \frac{h}{4}) + f(z_0 + hf(z_0 + \frac{h}{4}f(z_0, t_0), t_0+ \frac{h}{4}) \\ - \frac{h}{4}f(z_0,t_0), t_0 + \frac{3h}{4})),
\end{multline*}
and
$$
v_1(h)  =  v_0 + 2(f(z_0 + hf(z_0 + \frac{h}{4}f(z_0,t_0), t_0+ \frac{h}{4}) - \frac{h}{4}v_0,t_0 + \frac{3h}{4} ) - f(z_0 + \frac{h}{4}f(z_0,t_0), t_0+ \frac{h}{4})).
$$
Writing down the Taylor expansion in $h$ for $(z_1(h), v_1(h))$ we find exactly the same terms as in $(z(t_0 +h), v(t_0 + h))$ up to terms of the third order $O(h^3)$. The computations to obtain the Taylor expansion of $(z_1,v_1)$ were done using Maple software. This implies that the local error of ALF2 is of the 3rd order, and therefore, the global error is of order 2. This \textcolor{black}{completes} the proof.

\section{Numerical method for the adjoint} \label{sec.adjoint}
The expression of \eqref{eq:adjoint.discrete} for $\phi_{k+1}$ obtained from $\phi_{k}$ by a step forward of ALF2 can be interpreted as a rescaled step backward of ALF2 applied to the state-adjoint dynamics. Let us introduce a map $W_h$ depending on $h$, which acts on $(z,v,\lambda^z, \lambda^v)$ as follows. It only transforms $\lambda^v$  multiplying it by $-\frac{h^2}{16}$, that is
$$ 
W_{\textcolor{black}{\alpha}} (z, v, \lambda^z, \lambda^v) = \begin{pmatrix}
\mathrm{Id} & 0 & 0 & 0 \\
0 & \mathrm{Id} & 0 & 0 \\
0 & 0 & \mathrm{Id} & 0 \\
0 & 0 & 0 & \textcolor{black}{\alpha}\mathrm{Id}
\end{pmatrix}\begin{pmatrix}
z \\
v \\
\lambda^z \\
\lambda^v
\end{pmatrix}.
$$
%First, let us determine the formula \eqref{eq:adjoint.discrete} for ALF method and from it we will obtain the formula for Yoshida composition methods of any order.
%Adjoint method leads to explicit formulas for gradient computation using the step backward of the considered numerical methods. For the method to provide exact gradients, it is important to obtain the right discrete adjoint equation \eqref{eq:adjiont}. To obtain the discretization method for the adjoint, we first obtain the formula for ALF method. And then show, what changes, when we used compositions of it. 
\begin{Theorem} \label{th:adjoint:ALF}
The step backward of the discretized state-adjoint system %with variables $(z,v, \lambda^z, \lambda^v)$ 
associated to the ALF2 method satisfies 
\begin{equation} \label{eq:step.back.adjoint}
(z_k, v_k, \lambda^z_k, \lambda^v_k) = W_{\textcolor{black}{\frac{-h^2}{16}}}^{-1} \circ \Psi^{ALF2}_{-h} \circ W_{\textcolor{black}{\frac{-h^2}{16}}}(z_{k+1}, v_{k+1}, \lambda^z_{k+1},\lambda^v_{k+1}), 
\end{equation}
where $\Psi^{ALF2}_{-h}$ is applied \chris{to} the state-adjoint equations of the augemented system for $\phi = (z,v)$
$$\dot  \phi(t) = \tilde f(\phi(t), \theta(t)), \quad \dot \lambda = - \frac{\partial}{\partial \phi} \tilde f(\phi(t), \theta(t))^\top \lambda, $$
with $\tilde f(\phi, \theta) = (f(z,\theta), \frac{\partial }{\partial z}f(z,\theta) f(z,\theta) )$.
\end{Theorem}
\begin{proof}
In order to find the expression for ALF2, we first determine the expression for ALF and use the chain rule. Let us  compute $\frac{\partial \phi_{k+1}}{\partial \phi_k}$ for ALF method, where $\phi_{k+1} = (z_{k+1}, v_{k+1})$ and $\phi_{k} = (z_{k}, v_{k})$. Differentiating \eqref{eq:ALF.forward} with respect to $(z_k,v_k)$, we obtain
\begin{equation*}
\frac{\partial \phi_{k+1}}{\partial \phi_k} =
\begin{pmatrix}
	\mathrm{Id} + h\frac{\partial f}{\partial z}(z_{k} + \frac{h}{2}v_{k}, t_k + \frac{h}{2}) & \frac{h^2}{2 }\frac{\partial f}{\partial z}(z_{k} + \frac{h}{2}v_{k}, t_k + \frac{h}{2}) \\[0.3cm]
	2 \frac{\partial f}{\partial z}(z_{k} + \frac{h}{2}v_{k}, t_k + \frac{h}{2}) & h\frac{\partial f}{\partial z}(z_{k} + \frac{h}{2}v_{k}, t_k + \frac{h}{2}) - \mathrm{Id}
\end{pmatrix}.
\end{equation*}
This implies 
\begin{equation} \label{eq:adjoint.discrete.ALF}
\begin{aligned}
	\lambda^z_k &= \left( \mathrm{Id} + h \frac{\partial}{\partial z}f(z^{k} + \frac{h}{2}v^{k}, t^k + \frac{h}{2}) \right) \lambda^z_{k+1} + 2  \frac{\partial}{\partial z}f(z^{k} + \frac{h}{2}v^{k}, t^k + \frac{h}{2}) \lambda^v_{k+1}, \\
	\lambda^v_k &= \frac{h^2}{2} \frac{\partial}{\partial z}f(z^{k} + \frac{h}{2}v^{k}, t^k + \frac{h}{2}) \lambda^z_{k+1} + \left( h \frac{\partial}{\partial z}f(z^{k} + \frac{h}{2}v^{k}, t^k + \frac{h}{2}) - \mathrm{Id} \right) \lambda^v_{k+1} .
\end{aligned}
\end{equation}
Notice that \eqref{eq:adjoint.discrete.ALF} can be equivalently written as
\begin{equation} 
\begin{aligned}
	\lambda^z_k &= \lambda^z_{k+1} + h  \frac{\partial}{\partial z}f(z^{k} + \frac{h}{2}v^{k}, t^k + \frac{h}{2}) \left(\lambda^z_{k+1} + \frac{2}{h}\lambda^v_{k+1} \right), \\
	\lambda^v_k &= -2 \frac{\partial}{\partial z}f(z^{k}+ \frac{h}{2}v^{k}, t^k + \frac{h}{2}) \left(-\frac{h^2}{4}\lambda^z_{k+1} - \frac{h}{2}\lambda^v_{k+1}\right)  -\lambda^v_{k+1} .
\end{aligned}
\end{equation}
Let us now introduce $\tilde{\lambda}^v_k = -\frac{4}{h^2}\lambda^v_{k}$. Then equations take the following form
\begin{equation}  \label{eq.adjoint.ALF}
\begin{aligned}
	\lambda^z_k &= \lambda^z_{k+1} + h  \frac{\partial}{\partial z}f(z^{k} + \frac{h}{2}v^{k}, t^k + \frac{h}{2}) \left(\lambda^z_{k+1} - \frac{h}{2}\tilde{\lambda}^v_{k+1} \right), \\
	\tilde{\lambda}^v_k &= -2 \frac{\partial}{\partial z}f(z^{k} + \frac{h}{2}v^{k}, t^k + \frac{h}{2}) \left(\lambda^z_{k+1} - \frac{h}{2}\tilde{\lambda}^v_{k+1}\right)  -\tilde{\lambda}^v_{k+1} .
\end{aligned}
\end{equation}
Taking into account that $z_{k} + \frac{h}{2}v_{k} = z_{k+1} - \frac{h}{2}v_{k+1}$ from the construction of \eqref{eq:ALF.forward}-\eqref{eq:ALF.backward}, we conclude that variables $(\lambda^z_k, -\frac{4}{h^2}\lambda^v_{k})$ follow the backward integration with ALF method and its step backward defined by \eqref{eq:ALF.backward} applied to the continuous equations of the adjoint \eqref{eq:adjiont}. As a result, the step backward of the adjoint variables $\lambda$ can be expressed as
\begin{equation*}
(\lambda^z_k, \lambda^v_k) = \widehat W_{\textcolor{black}{\frac{-h^2}{4}}}^{-1} \circ \widehat \Psi^{ALF}_{-h}(z_{k+1},v_{k+1})\circ \widehat W_{\textcolor{black}{\frac{-h^2}{4}}} (\lambda^z_{k+1}, \lambda^v_{k+1}),
\end{equation*}
\textcolor{black}{where $\widehat W_\alpha$ is a projection of $W_\alpha$ to variables $(\lambda^z_k, \lambda^v_k)$ and $\widehat \Psi^{ALF}_{-h}(z_{k+1},v_{k+1})$ } stands for a projection of the backward ALF step to $(\lambda^z, \lambda^v)$, which is still a function of $(z_{k+1},v_{k+1})$.
To deduce the formula for \chris{the} ALF2 method, we use its composition structure, namely, $\Psi^{ALF2}_h = \Psi^{ALF}_{h/2} \circ \Psi^{ALF}_{h/2}$. This implies
\begin{equation*}
\frac{\partial \phi_{k+1}}{\partial \phi_k} 
=  \frac{\partial }{\partial \phi_k} \left( \Psi^{ALF}_{h/2} \circ \Psi^{ALF}_{h/2} \right) =  \left( \frac{\partial \Psi^{ALF2}_{h/2} }{ \partial \phi}(\phi_{k+\frac 12}) \right)  \circ\left( \frac{\partial \Psi^{ALF2}_{h/2} }{ \partial \phi}( \phi_{k}) \right) 
\end{equation*}
with
$$
\phi_{k+\frac 12} = \Psi^{ALF2}_{h/2}(\phi_{k}) = \Psi^{ALF2}_{-h/2}(\phi_{k+1}).
$$
As a result,

\begin{align*}
\left( \frac{\partial \phi_{k+1}}{\partial \phi_k} \right)^\top
&=  \left( \frac{\partial \Psi^{ALF2}_{h/2} }{ \partial \phi}( \phi_{k}) \right)^\top \circ \left( \frac{\partial \Psi^{ALF2}_{h/2} }{ \partial \phi}(\phi_{k+\frac 12}) \right)^\top   \\
&= \textcolor{black}{\widehat W_{\frac{-h^2}{16}}^{-1}\circ \widehat \Psi^{ALF}_{-h/2}(\phi_{k+\frac 12})\circ\widehat W_{\frac{-h^2}{16}}\circ\widehat W_{\frac{-h^2}{16}}^{-1}\circ\widehat \Psi^{ALF}_{-h/2}(\phi_{k+1})\circ\widehat W_{\frac{-h^2}{16}}} \\
&= \textcolor{black}{\widehat W_{\frac{-h^2}{16}}^{-1}\circ\widehat \Psi^{ALF}_{-h/2}(\Psi^{ALF2}_{-h/2}(\phi_{k+1}))\circ \widehat\Psi^{ALF}_{-h/2}(\phi_{k+1})\circ\widehat W_{\frac{-h^2}{16}}} \\
&= \textcolor{black}{\widehat W_{\frac{-h^2}{16}}^{-1}\circ\widehat\Psi^{ALF2}_{-h}(z_{k+1},v_{k+1})\circ\widehat W_{\frac{-h^2}{16}}.}
\end{align*}  
%The map $\widetilde W_{h/2}$ rescales $\lambda^v$ to $-\frac{h^2}{16}\lambda^v$. We introduce $W_h$ for the transformation of state-adjoint variables, which leaves all variables unchanged except $\lambda^v$, which is mapped to $-\frac{h^2}{16}\lambda^v$. 
The resulting equations for the backward step of the state-adjoint system are given in \eqref{eq:step.back.adjoint}. This \textcolor{black}{completes} the proof of Theorem~\ref{th:adjoint:ALF}.
\end{proof}

The formula for the adjoint of Yoshida methods $\Phi^Y_{2k}$ follows from the composition structure of the method and is presented the following theorem. \textcolor{black}{With a slight abuse of notation we denote $\Phi^{ALF2}$ by $\Phi^Y_{2}$.}
\begin{Theorem} \label{th:composition.grad}
Assume that the discrete one step method in \eqref{eq:neural.ODE} is given \textcolor{black}{for $k \geq 2$} by
\begin{equation} \label{eq:composition}
\Psi^Y_{2k}(h) = \Psi^{Y}_{2k-2}(ah) \circ \Psi^{Y}_{2k-2}(bh) \circ \Psi^{Y}_{2k-2}(ah), \qquad \phi_{k+1} = \Psi^Y_{2k}(h) \circ \phi_k.
\end{equation}
%
%such that $z_{k+1} = \Psi(z_k)$, 
Then the state-adjoint backward step can be computed recursively as follows 
\begin{equation} \label{eq.adjoint.Yoshida}
(\phi_k, \lambda_k) = \widetilde \Psi^{Y}_{2k-2}(ah) \circ \widetilde \Psi^{Y}_{2k-2}(bh) \circ \widetilde \Psi^{Y}_{2k-2}(ah)(\phi_{\textcolor{black}{k+1}}, \lambda_{\textcolor{black}{k+1}}), 
\end{equation}
with $\widetilde \Psi^{Y}_{2k-2}$ the map, which defines the backward step of state-adjoint system of the method $\Phi^Y_{2k-2}$.
%is a composition method of the following form
%$$ \lambda = \lambda^{\alpha_1 h }\circ \lambda^{\alpha_2 h } \circ \cdots \circ \lambda^{\alpha_K h },$$
%which means that $\lambda_k$ is defined from $\lambda_{k+1}$ by 
%\begin{equation} \label{eq.adjoint.Yoshida}
%\begin{aligned}
%\lambda^{1}_{k} & = \left( \frac{\partial \Psi^{Y}_{2k-2}(ah) }{ \partial \phi}(z_{k+1}) \right)^\top \lambda_{k+1}, \\
%\lambda^{2}_{k} &= \left( \frac{\partial \Psi^{Y}_{2k-2}(bh) }{ \partial \phi}((\Psi^{Y}_{2k-2}(ah))^{-1}\circ z_{k+1} ) \right)^\top \lambda^{1}_{k}, \\
%\lambda^{2}_{k} &= \left( \frac{\partial \Psi^{Y}_{2k-2}(bh) }{ \partial \phi}((\Psi^{Y}_{2k-2}(ah))^{-1}(z_{k+1}) ) \right)^\top \lambda^{\alpha_i h }_{k}, \\
%\lambda_{k} &=  \left( \frac{\partial \Psi^{Y}_{2k-2}(ah) }{ \partial \phi}( (\Psi^{Y}_{2k-2}(bh))^{-1}\circ (\Psi^{Y}_{2k-2}(ah))^{-1}\circ z_{k+1}) \right)^\top \lambda^{2}_{k}, 
%\end{aligned}
%\end{equation}
%where each of the equation is the adjoint equation of $\Phi^Y_{2k-2}$ and its expression can be obtained using the composition form \eqref{eq:composition}. 
\end{Theorem}
\begin{proof}
The proof is by induction on $k$ in the considered method $\Psi^Y_{2k}$ and is based on the composition structure of $\Psi^Y_{2k}$ in \eqref{eq:composition}. Let $k=2$, then 
$\Psi^Y_{4}(h) = \Psi^{ALF2}_{ah} \circ \Psi^{ALF2}_{bh} \circ \Psi^{ALF2}_{ah} $.
%For $s=1$ the statement is trivial. Assume that it holds for $S = s_0$, that is. Consider now a composition of $s_0 +1$ elements
%\begin{equation*}
%   \Psi = \underbrace{\Psi_1(\alpha_1h)\circ \cdots\circ\Psi_1(\alpha_{s_0}h)}_{\normalfont \bar \Psi }\circ\Psi_1(\alpha_{s_0+1}h).
%\end{equation*}
%The corresponding step forward can be written as
%$$
%\phi_{k+1} = \bar \Psi \circ \Psi_1(\alpha_{s_0+1}h) \circ \phi_k.
%$$
%
By \chris{the} chain rule we have
\begin{equation*} %\label{eq:adjoint.yoshida.ALF2}
\left(\frac{\partial \phi_{k+1}}{\partial \phi_k} \right)^\top
= \left( \frac{\partial \Psi^{ALF2}_{ah} }{ \partial \phi}( \phi_{k+\frac 13} ) \right)^\top \circ \left( \frac{\partial \Psi^{ALF2}_{bh} }{ \partial \phi}( \phi_{k+\frac 23}) \right)^\top \circ \left( \frac{\partial \Psi^{ALF2}_{ah} }{ \partial \phi}(\phi_{k+1}) \right)^\top
\end{equation*}
with
$$
\begin{aligned}
\phi_{k+\frac 23} &= \Psi^{ALF2}_{-ah}(\phi_{k+1}), \\
\phi_{k+\frac 13} &= \Psi^{ALF2}_{-bh} \circ \Psi^{ALF2}_{-ah}(\phi_{k+1}).
\end{aligned}
$$
By construction of the backward step of the state adjoint system by ALF2 shown in \eqref{eq:step.back.adjoint}, we have
\begin{multline*}
\left(\frac{\partial \phi_{k+1}}{\partial \phi_k} \right)^\top 
= Pr_{\lambda}\left( W_{ah}^{-1}\circ\Psi^{ALF2}_{-ah}\circ W_{ah}\right) \circ Pr_{\lambda}\left(W_{bh}^{-1}\circ\Psi^{ALF2}_{-bh}\circ W_{bh}\right)\circ \\
\circ Pr_{\lambda}\left( W_{ah}^{-1}\circ\Psi^{ALF2}_{-ah}\circ W_{ah}\right),
\end{multline*}
where $\textcolor{black}{\widehat \Psi^{ALF2}_{h_i}} = Pr_{\lambda}\left( W_{h_i}^{-1}\circ\Psi^{ALF2}_{-h_i}\circ W_{h_i}\right)$, $h_i \in \{ ah, bh \}$ defines a step backward \chris{with the} ALF2 method with step-size $h_i$ in the adjoint variable. This proves the Theorem for $k=2$. Let us assume now that the statement of the theorem holds for $k = k_0$ and we consider the adjoint method for $\Psi^Y_{2k_0}(h) = \Psi^{Y}_{2k_0-2}(ah) \circ \Psi^{Y}_{2k_0-2}(bh) \circ \Psi^{Y}_{2k_0-2}(ah)$. As before, applying the chain rule and the assumption of the induction, it follows that 
\begin{align*} %\label{eq:adjoint.yoshida.ALF2}
\left(\frac{\partial \phi_{k+1}}{\partial \phi_k} \right)^\top
&= \left( \frac{\partial \Psi^{Y}_{2k_0-2}(ah) }{ \partial \phi}( \tilde \phi_{k+\frac 13} ) \right)^\top \circ \left( \frac{\partial \Psi^{Y}_{2k_0-2}(bh) }{ \partial \phi}(\tilde \phi_{k+\frac 23}) \right)^\top \circ \left( \frac{\partial \Psi^{Y}_{2k_0-2}(ah) }{ \partial \phi}(\phi_{k+1}) \right)^\top \\
&= \textcolor{black}{\widehat{\widetilde \Psi^{Y}}_{2k_0-2}(ah) \circ \widehat{\widetilde \Psi^{Y}}_{2k_0-2}(bh) \circ \widehat{\widetilde \Psi^{Y}}_{2k_0-2}(ah),}
\end{align*}
where we used the notation
$$
\textcolor{black}{ \tilde \phi_{k+\frac 23} = \Psi^{Y}_{2k_0-2}(-ah)(\phi_{k+1}), \quad \tilde \phi_{k+\frac 13} = \Psi^{Y}_{2k_0-2}(-bh) \circ \Psi^{Y}_{2k_0-2}(-ah)(\phi_{k+1}),}
$$
and \textcolor{black}{$\widehat{\widetilde \Psi^{Y}}_{2k_0-2}$} the projection of the step backward associated to the state-adjoint system and $\Psi^{Y}_{2k_0-2}$ method. This \textcolor{black}{completes} the induction step and the proof.
\end{proof}
\textcolor{black}{In case of $k=2$, Theorem \ref{th:composition.grad} in combination with \eqref{eq:step.back.adjoint} leads to the following expression 
$$ 
\widetilde \Psi^{Y}_{4} = W_{\frac{-(ah)^2}{16}}^{-1} \circ \Psi^{ALF2}_{-ah} \circ W_{\frac{a^2}{b^2}} \circ \Psi^{ALF2}_{-bh} \circ W_{\frac{b^2}{a^2}} \circ \Psi^{ALF2}_{-ah} \circ W_{\frac{-(ah)^2}{16}}. 
$$
}
%Notice that all the components in the composition $\Psi^Y_{2k}(h)$ are reversible and, therefore,  $(\Psi^Y_{2k}(h))^{-1}$ can be computed by step backward as described in Algorithm~\ref{alg:step.backward}. 
The obtained results lead to the Algorithm~\ref{alg:reconstruction2} for the computation of gradients.

%\textcolor{black}{SO: proofs in appendix?} \sofya{Yes, done}
\begin{algorithm}
\caption{Computation of gradients}\label{alg:reconstruction2}
%\FJi{Est-ce que \c{c}a a du sens de mettre cet algo ici? Je ne vois pas commnt on peut faire l'\'{e}tape 1 algorithmiquement. Ce serait mieux il me semble de dire qu'on se place sur $\mathcal{J} \setminus \mathcal{J}_p$ puisqu'on montrera que c'est dense.}
\begin{algorithmic}
\State 1. Input: training data $z_0$, initialization of parameters $\theta$, velocity $v_0 = f(z_0,\theta_0)$
\State 2. Propagate through the network using $\Psi^Y_{2k}$ to get $(z_N, v_N)$
\State 3. Set $\lambda^z_N = \nabla L(z_N)$ and $\lambda^v_N = 0$
\For{\texttt{k = N to 1}}
\State 4. Compute $\phi_k, \lambda_k$ from $\phi_{k+1}, \lambda_{k+1}$ using \textcolor{black}{\eqref{eq:step.back.adjoint} and \eqref{eq.adjoint.Yoshida}.}%Theorem~\ref{th:composition.grad}.
\State 5. Compute $\frac{\partial J(\theta)}{\partial \theta_{k}}$ using \textcolor{black}{\eqref{eq:adjoint.grad.discrete}}
\EndFor

\State 6. Output: gradients  $\frac{\partial J(\theta)}{\partial \theta_{k}}$ for $k = 1, \dots N-1$. 
%	\EndIf
\end{algorithmic}
\end{algorithm}

\section{Details of numerical experiments} 
In all the numerical experiments, our implementation of the  Yoshida composition method uses the code of the MALI network \cite{zhuang2021mali}. We use the steps forward and backward of the ALF method as composition steps to compute ALF2 and its Yoshida composition.  
\subsection{Kepler problem} \label{sec:Kepler}
The training data for the comparison of the computational time in Table~\ref{table:Kepler_time} is given by a trajectory $x$ of \eqref{eq:Kepler} with initial condition $x_0 = (0.75,  0, 0, \frac{0.9 \pi}{4}\sqrt{\frac{5}{3}})$ on time interval $[0,T] = [0,1]$, which is an elliptic orbit. The trajectory is obtained by numerical integration using \texttt{sci.integrate.odeint} with relative and absolute tolerances $10^{-7}$ and $10^{-8}$ respectively and maximum step size $10^{-5}$. %The trajectory is then used in the loss function $L =   \|x_\theta - x(T, x_0)\|^2$. 
The optimizer used in the training is \texttt{SGD} from \texttt{PyTorch} with initial learning rate 0.1 and scaled by 0.95 for each epoch. \textcolor{black}{For completeness, we show the evaluation of the parameter error across the learning displayed as a function of time in Figure~\ref{fig:Kepler.error.vs.time} and as a function of epochs in Figure~\ref{fig:Kepler.error.vs.epochs}. In the plots we show the results obtained with ALF, Y4 and also Runge-Kutta 4(5) (RK45), the latter is not a reversible method and requires storage of the intermediate states obtained during the integration forward. This implies additional memory consumption, namely, at each epoch the algorithm saves 8 additional states obtained during integration forward, making the memory consumption of the training higher. The four plots in Figures~\ref{fig:Kepler.error.vs.epochs} and \ref{fig:Kepler.error.vs.time} are obtained for different initializations of the parameters $\alpha_0$ in the learning, namely, $\alpha_0 = 1.3, 0.1, 0.7, 0.75$.}
\begin{figure}[h]
\centering
\includegraphics[width=0.35\textwidth]{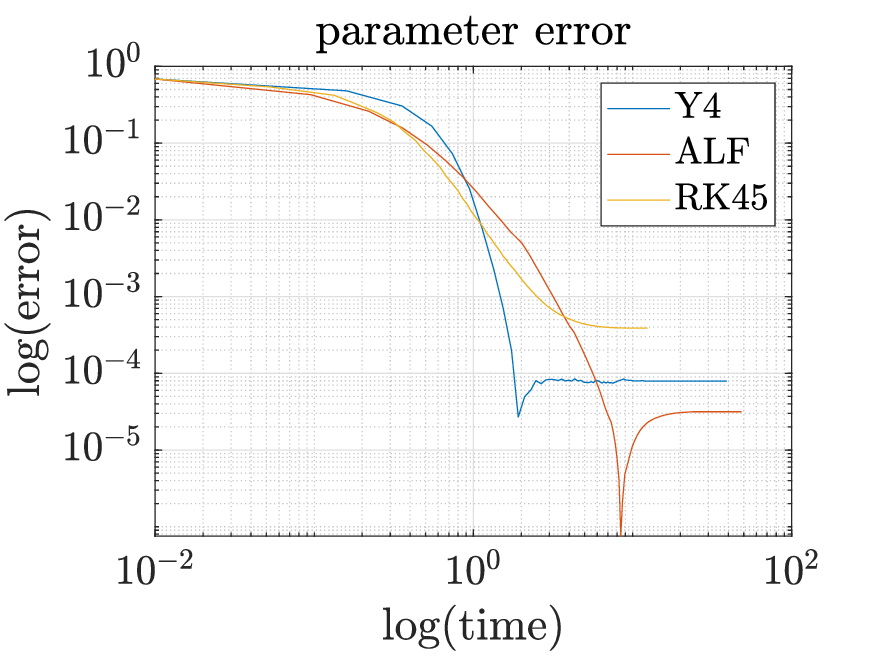} \ \ \ 
\includegraphics[width=0.35\textwidth]{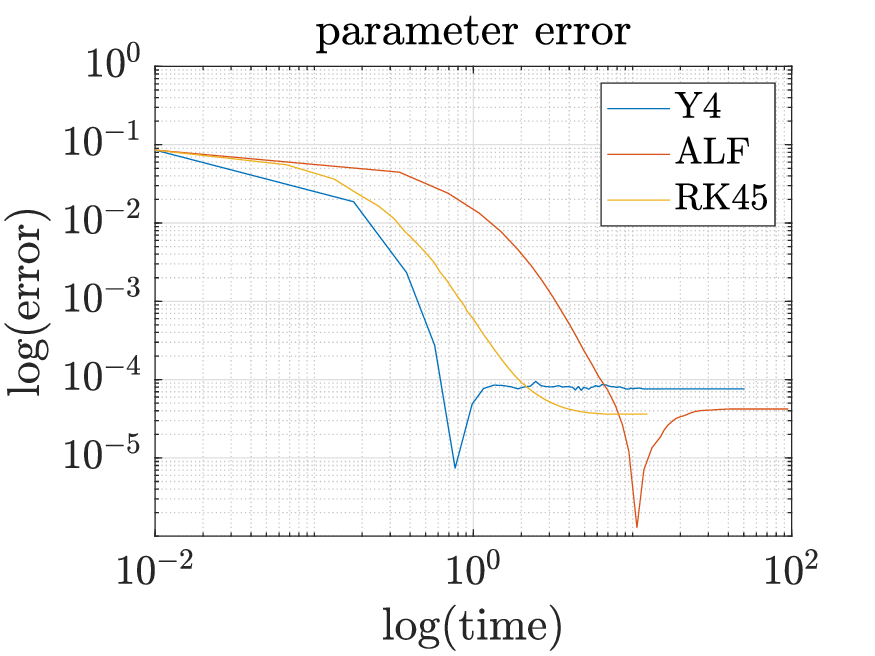} \\
\includegraphics[width=0.35\textwidth]{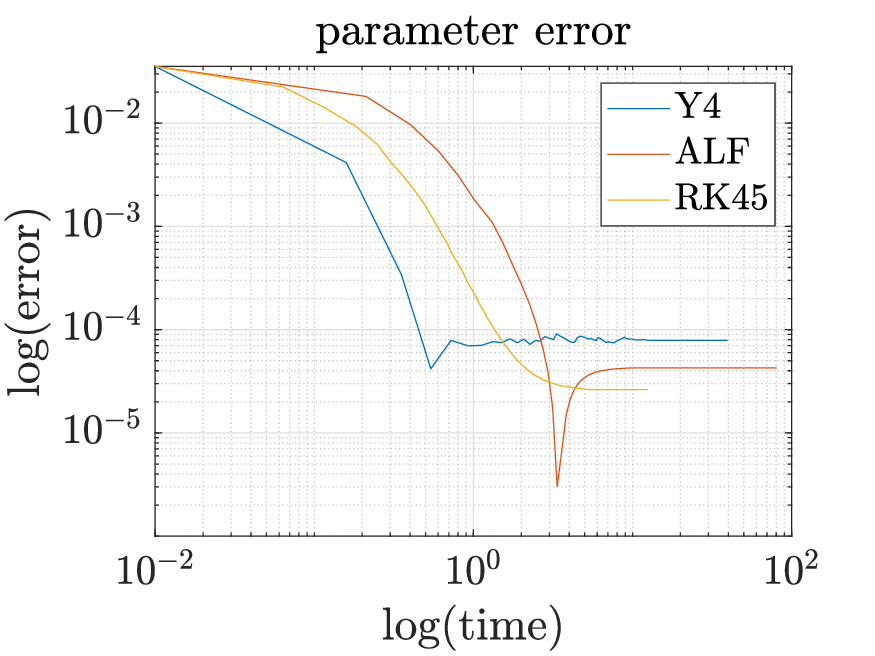} \ \ \ 
\includegraphics[width=0.35\textwidth]{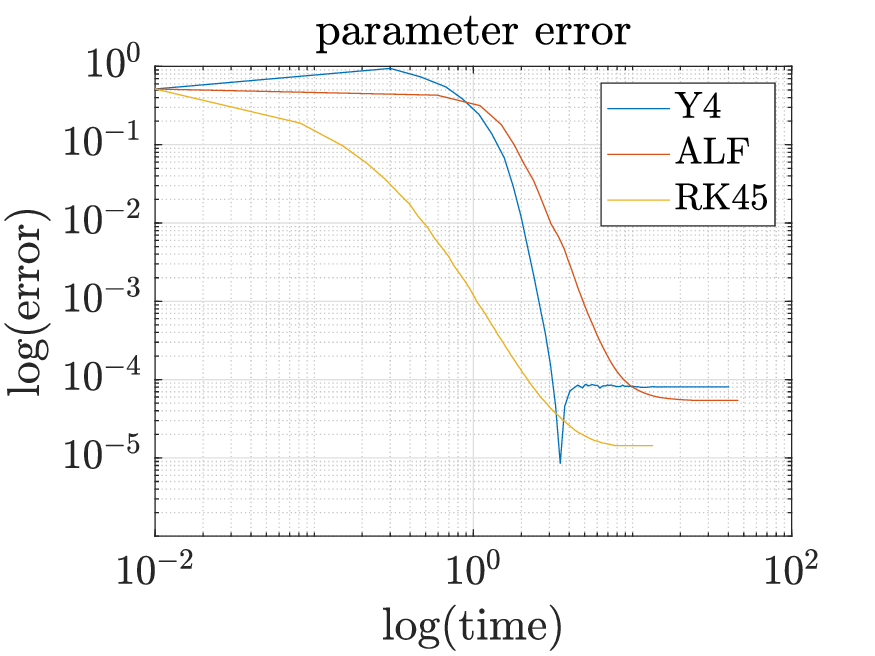}
\caption{Error of the learned parameter with respect to the ground truth $\alpha$ as a function of time.} 
\label{fig:Kepler.error.vs.time}
%\fbox{\rule[-.5cm]{0cm}{4cm} \rule[-.5cm]{4cm}{0cm}}
\end{figure}

\begin{figure}[h]
\centering
\includegraphics[width=0.35\textwidth]{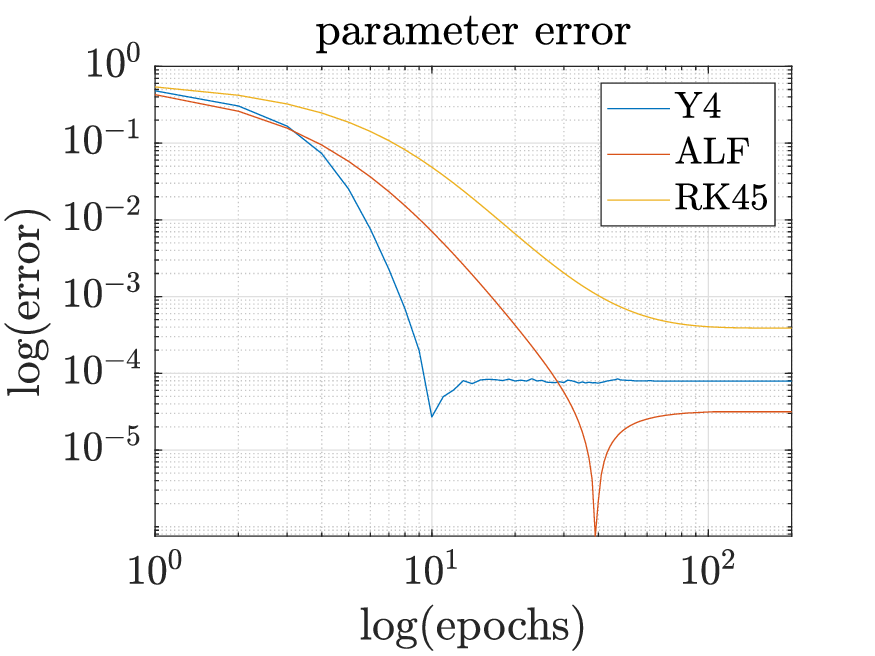} \ \ \ 
\includegraphics[width=0.35\textwidth]{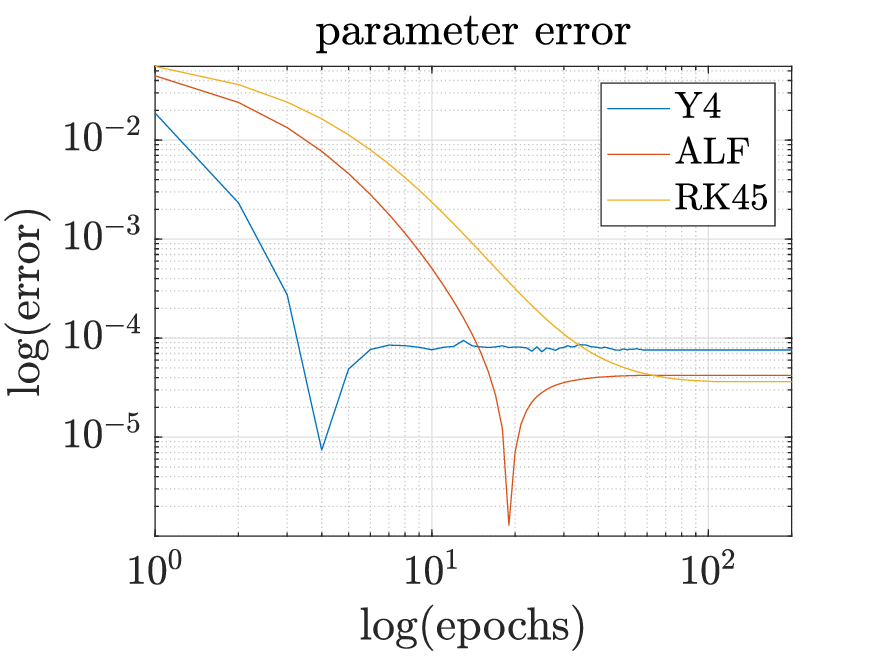} \\
\includegraphics[width=0.35\textwidth]{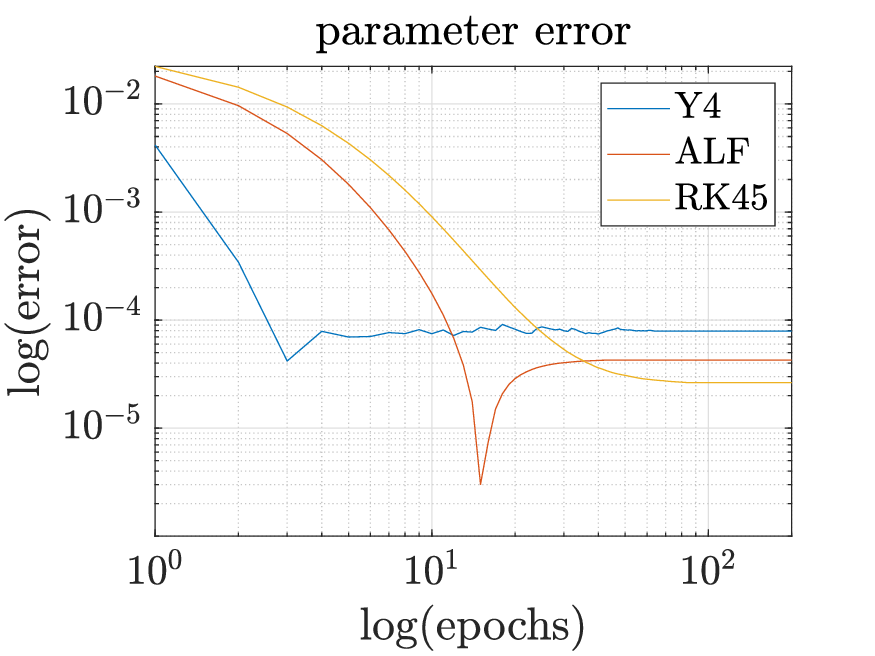} \ \ \ 
\includegraphics[width=0.35\textwidth]{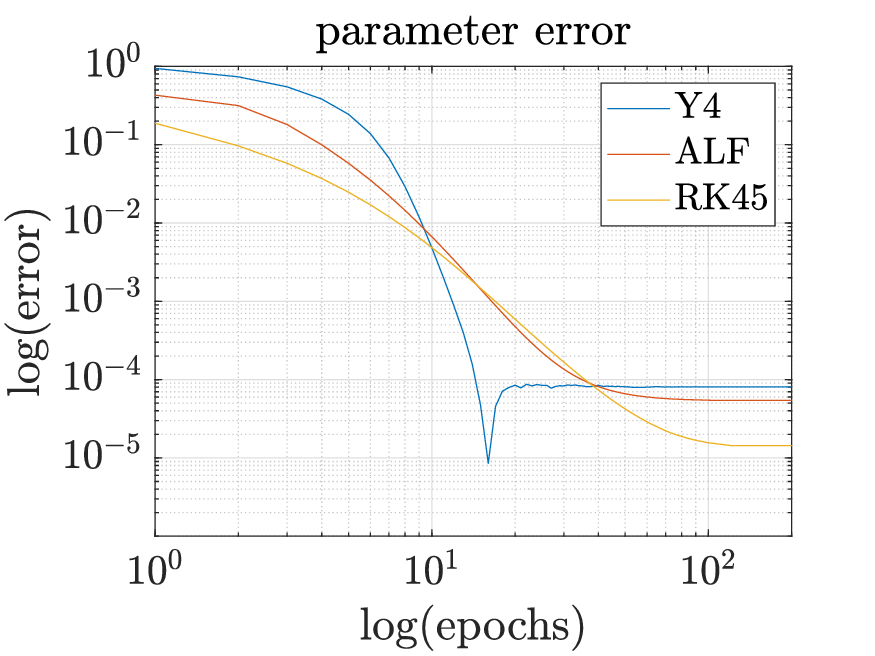}
\caption{Error of the learned parameter with respect to the ground truth $\alpha$ as a function of epochs.} 
\label{fig:Kepler.error.vs.epochs}
%\fbox{\rule[-.5cm]{0cm}{4cm} \rule[-.5cm]{4cm}{0cm}}
\end{figure}

The error landscape in Figure~\ref{fig:Error.landscape} is obtained by considering 81 trajectories obtained using the same integration method as for the time comparison explained above for 81 different initial conditions $(x_0)_i$ in a neighborhood of $x_0$, given by a 4-dimensional box of diameter 0.4 around $x_0$. The points $(x_0)_i$ are chosen on a grid with a step size $0.1$, which includes $x_0$ as its point. %\chris{[Are they sampled or obtained from a Halton sequence?]}. 
\textcolor{black}{The behaviour of ALF and Y4 with adaptive stepping can be better understood when looking at fixed step methods, when the step size $h_i = h$ is fixed for all the steps. The loss landscape visualized in Figure~\ref{fig:Error.landscape} for fixed step ALF and fixed step Y4 shows that the minimum value of the loss is achieved at a better precision of the true parameter for the higher order methods than for the lower order method,
\textcolor{black}{which will be explained in more detail below.}
%This phenomenon explains, that for a higher accuracy, the lower order method has to choose smaller step sizes, which implies that a large number of steps is needed and this makes the integration time larger in comparison with a higher order method. This can be an explanation, why Y4 is faster in both epoch computations and also to get to the desired accuracy in the training.
}
The loss visualized in Figure~\ref{fig:Error.landscape} as a function of $\alpha$ is \[L(\alpha_k) =  \frac{1}{81} \sum_{i=1}^{81} \sum_{j=1}^{5} \sum  \|(q_{N_j}({\alpha_k}))_i - q(t_j, (x_0)_i)\|^2\] with $q_{N_j}$ projection of $z_{N_j}$ to $q$-coordinate and $\alpha_k$ taking 300 values in  $[\frac{\pi}{4} - 10^{-4}, \frac{\pi}{4} + 10^{-4}]$.
\textcolor{black}{Here $(q_{N_j}(\alpha))_i$ is obtained by numerical integration of \eqref{eq:Kepler}.}

\begin{figure}[h] 
%\vspace{-2mm}
\centering
\includegraphics[width=0.7\textwidth]{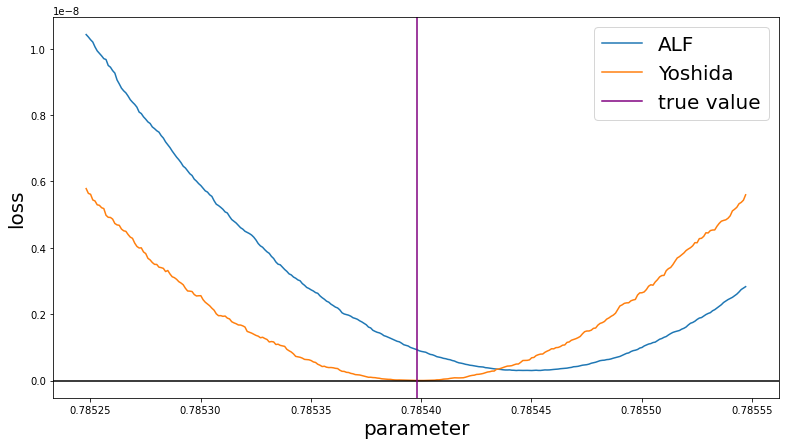} 
%\vspace{-3mm}
\caption{Error landscape of ALF and Y4 methods for Kepler problem showing the loss computed for the parameters in a neighbourhood of the true value of $\alpha$ displayed by a vertical line.} 
\label{fig:Error.landscape}%\textcolor{black}{more explanation in caption}}
\end{figure}
%\chris{[Is the factor 25 in the loss function a typo?]}

\textcolor{black}{
%Indeed, figure \ref{fig:Error.landscape} visualizes the benefit of high order methods for accurate parameter identification. Indeed,
If $(q_{N_j}(\alpha))_i$ was obtained by exact integration of \eqref{eq:Kepler} and in the absence of noise and round-off errors, true parameter values constitute minima for $L$. We interpret the application of a numerical integrator as a perturbation of size $\mathcal{O}(h^p)$ to the exact $(q_{N_j}(\alpha))_i$, where $h$ is the step size of the integration and $p$ the order of the numerical method. This yields a perturbation $\tilde L$ of $L$ of size $\mathcal{O}(h^{2p})$ in case of the mean-square loss. Thus, assuming that the local minima of $L$ at the true parameter value is non-degenerate, $\tilde L$ has a local minimum within a ball around the true parameter of size $\mathcal{O}(h^p)$. This follows from classical discussions on the numerical conditioning of computing zeros of a function as, for instance, in \cite[§5.2]{Dahmen2022}.
This provides a direct relation of the order of an integration method and the accuracy of identified parameters.}

\textcolor{black}{Notice that in the adaptive-step size context the perturbation of $L$ and, thus, the error of its minima are controlled by the provided error tolerance. However, the discussion shows that in order to be able to expect the same accuracy in the parameter identification, neural ODEs based on lower-order methods require more integration steps than neural ODEs based on high-order methods.}

\textcolor{black}{The above $\mathcal{O}(h^p)$ error relation in the parameter estimation constitutes an asymptotic upper bound. In Geometric Numerical Integration errors of numerical integrators can enter in highly symmetric way \cite{Hairer2013}. In symplectic integration of Hamiltonian systems, for instance, energy errors enter in an unbiased form. If the sought parameter is related to the geometric structure that is preserved by the geometric numerical integrator, parameters can potentially be estimated to higher accuracy than expected by the order of the numerical integrator. This, together with backward error analysis techniques, was used in \cite{Offen2022}, for instance, to accurately identify a Hamiltonian function of a dynamical system even though a low order method was used to discretize the dynamical system.
These techniques, however, are tailored to the geometric problem at hand, while the approach of this article considers a more general case.%, in which we do not expect any symmetry effects generically.
}

\subsection{Nonlinear harmonic oscillator}
There are two settings considered for the learning of the dynamics \eqref{eq:coupled_osc}. In the first setting, the learning problem is the parameter identification as presented in Section~\ref{sec:nonlin.osc.param}. In the second case we consider the parametrization of the potential by a neural network as described in Section~\ref{sec:nonlin.osc.nn}. Here we give more details on both problems.

\subsubsection{Identification of parameters} \label{sec:annex.nonlin.osc}
In the experiments for the time comparison shown in Table~\ref{table:Osc1} we consider a set of $200$ trajectories in the training data with the initial conditions generated by the Halton sequence in a 20-dimensional box around zero vector $x_0$ with diameter 2.0. The trajectories are obtained by numerical integration using \texttt{sci.integrate.odeint} with relative and absolute tolerances $10^{-13}$ and $10^{-14}$ respectively. In the training we use \texttt{AdamW} optimizer from \texttt{PyTorch} with learning rate scheduler \texttt{ExponentialLR}. The results shown in Table~\ref{table:Osc1} are obtained with different learning rates, namely, the first two with the initial learning rate $10^{-2}$ and $\gamma = 0.995$, the last three with the initial learning rate $10^{-1}$ and $\gamma = 0.998, 0.997, 0.99$ for the tree results respectively. At each epoch we consider all 200 trajectories, so that the loss is $L =  \frac{1}{200} \sum_{i=1}^{200} \|(z_N)_i - x(T, (x_0)_i)\|^2$ with $T = 0.5$. While Table~\ref{table:Osc1} compares the training time of ALF and Y4, it is also important to compare their performance in the learned parameters. In Figure~\ref{fig:Osc.error.vs.time} we show the results in the error of the learned parameters as a function of computational time measured at each epoch of ALF, Y4 and also RK45, which is not reversible. We can see that Y4 in not only faster than ALF in the training but the same also holds for the error in the learned parameters. While RK45 is the fastest to get to accurate parameters, it also requires the storing of 80 additional states during integration at each epoch, which means a considerable contribution to the memory costs. To better understand the reasons of the faster learning of Y4 than ALF, we show in Figure~\ref{fig:Osc.time} the computation time accumulated at each epoch of the training. The computational time per epoch is smaller for Y4, which contributes to the faster convergence in the training. In both Figures~\ref{fig:Osc.error.vs.time} and\ref{fig:Osc.time}, the four plots correspond to different random initializations of the parameters in the optimization.

\begin{figure}[h] 
\centering
\includegraphics[width=0.35\textwidth]{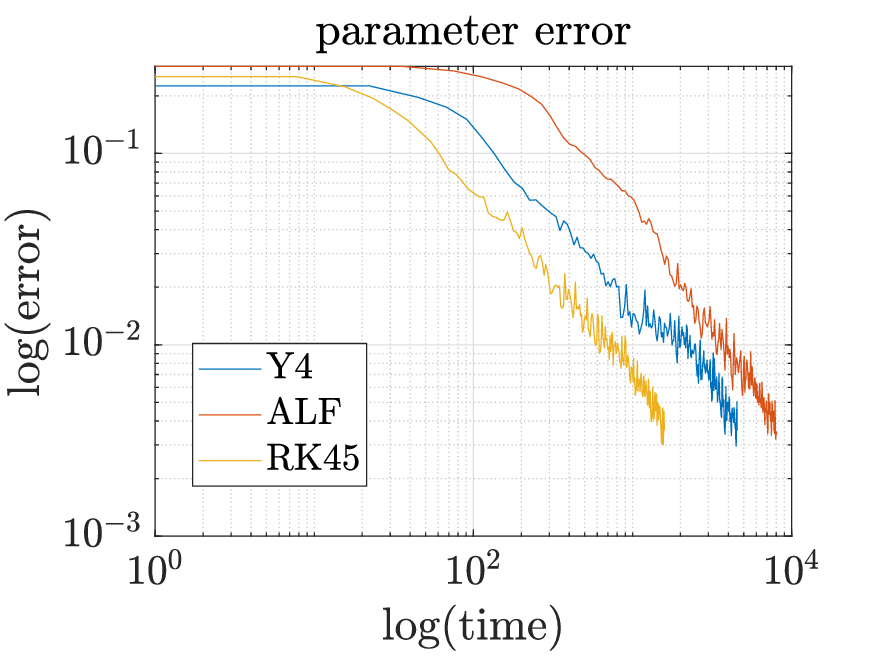} \ \ \ 
\includegraphics[width=0.35\textwidth]{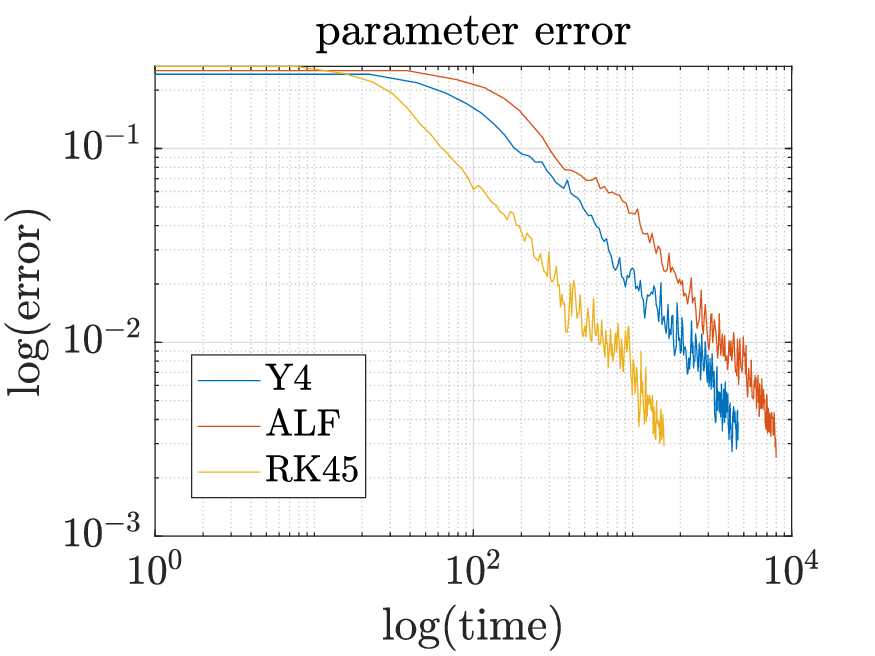} \\
\includegraphics[width=0.35\textwidth]{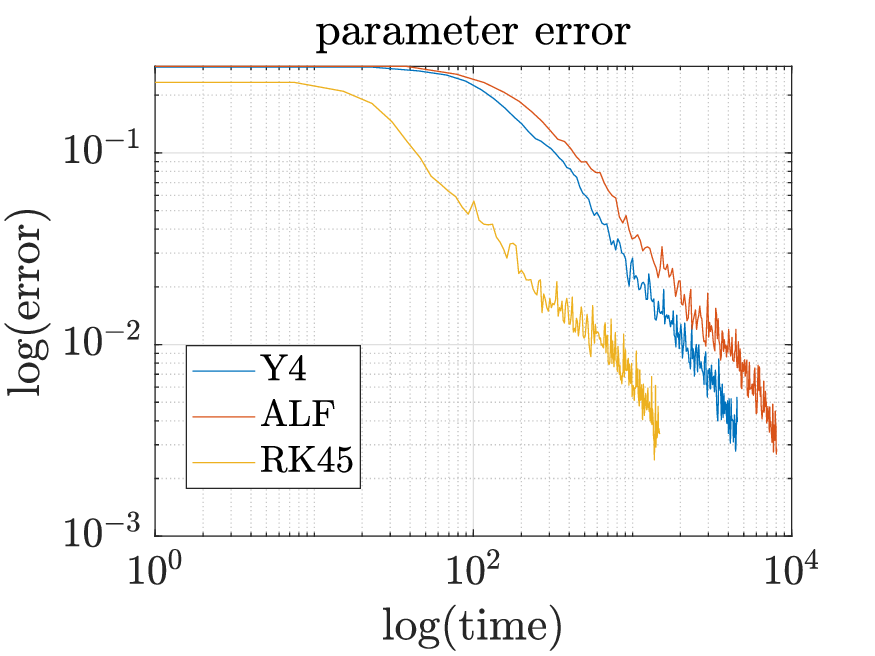} \ \ \ 
\includegraphics[width=0.35\textwidth]{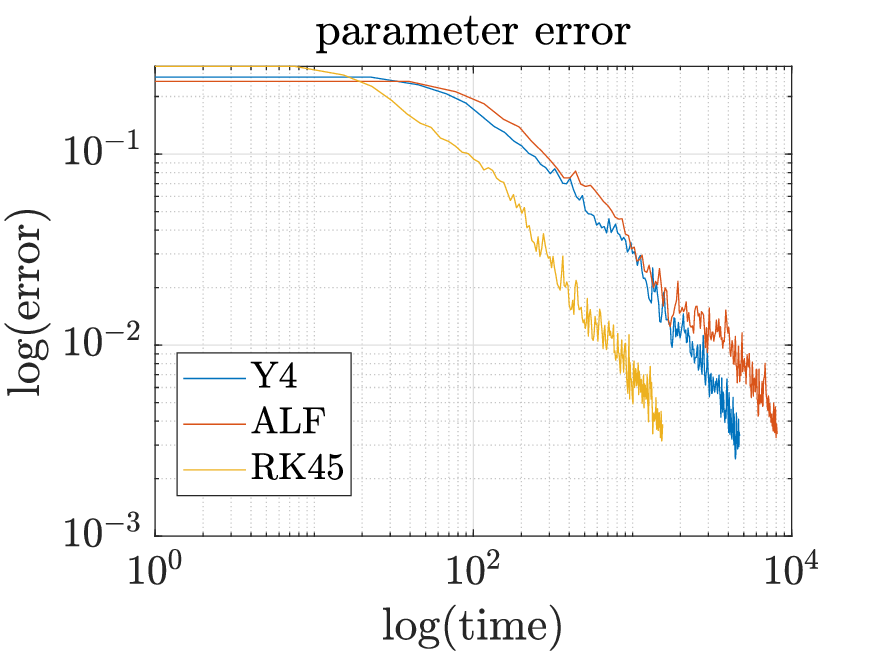}
\caption{Error in learned parameters as a function of time.} 
\label{fig:Osc.error.vs.time}
%\fbox{\rule[-.5cm]{0cm}{4cm} \rule[-.5cm]{4cm}{0cm}}
\end{figure}

\begin{figure}[h] 
\centering
\includegraphics[width=0.35\textwidth]{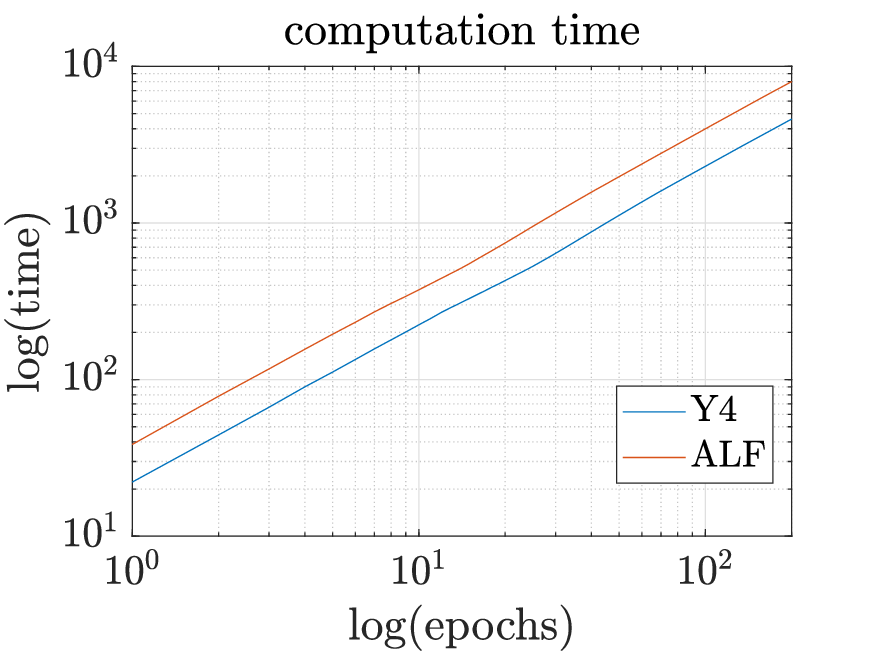} \ \ \ 
\includegraphics[width=0.35\textwidth]{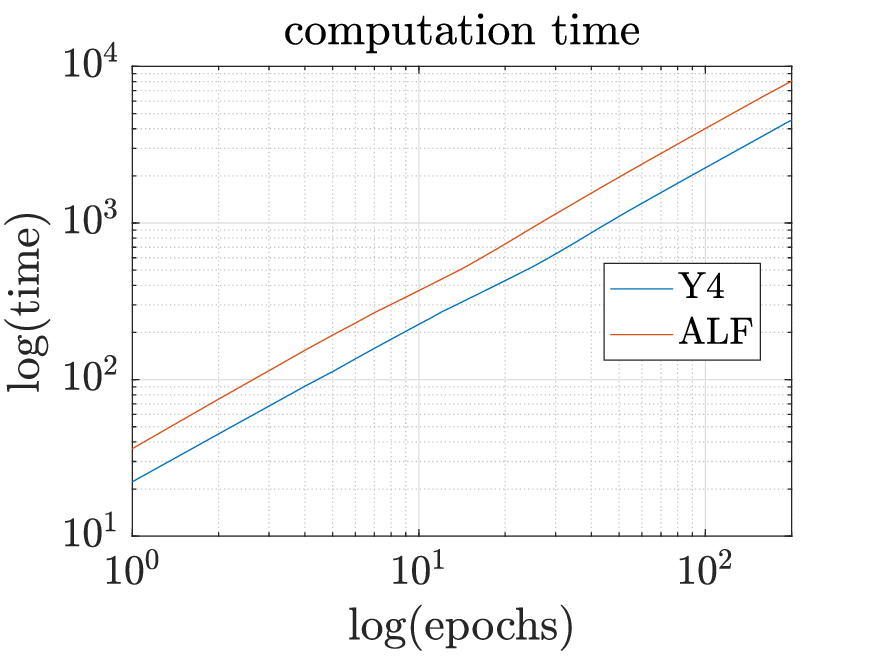} \\
\includegraphics[width=0.35\textwidth]{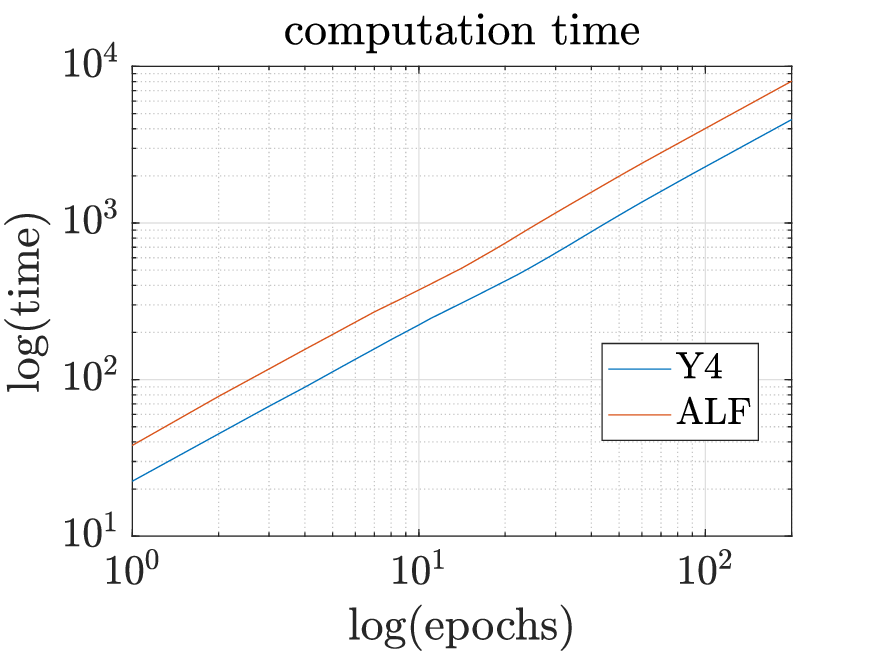} \ \ \ 
\includegraphics[width=0.35\textwidth]{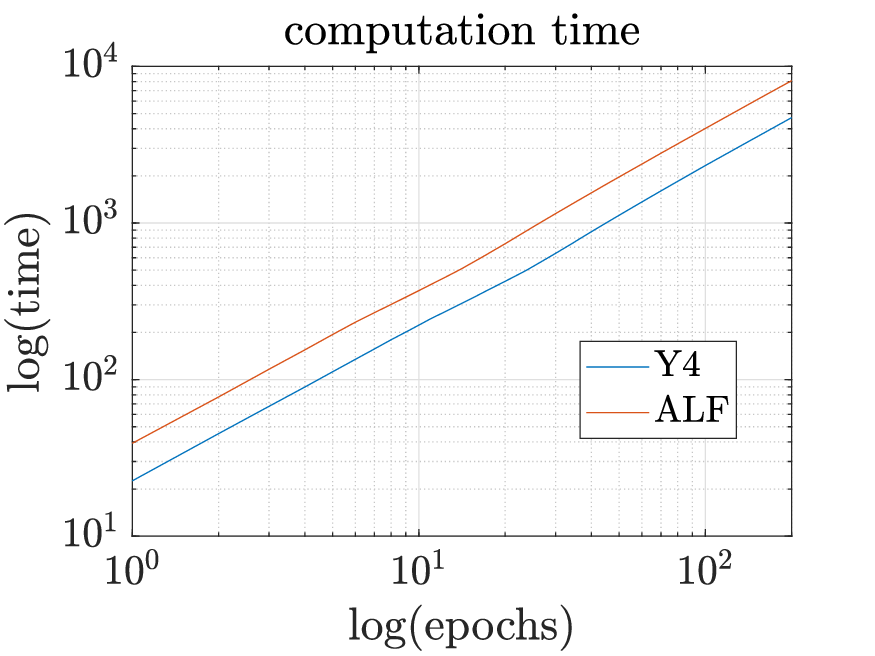}
\caption{Time of computation in function of epochs. When the curve is positioned lower, the corresponding algorithm is faster.}  
\label{fig:Osc.time}
%\fbox{\rule[-.5cm]{0cm}{4cm} \rule[-.5cm]{4cm}{0cm}}
\end{figure}

\textcolor{black}{In addition to results obtained for adaptive stepping, we test ALF and Y4 with the step size fixed to $h = 0.1$ and the training until either the training accuracy reaches $10^{-4}$ or the number of epochs reaches 500. Figure~\ref{fig:ALF-Y4-training-loss} shows that ALF is stuck at the training accuracy $10^{-2}$ and the training stops because of reaching 500 epochs, while Y4 converges to accuracy $10^{-4}$ with $181$ epochs. The same behaviour is observed for different parameter initialization. Decreasing the step size to $h=0.01$ permits ALF to reach accuracy $10^{-4}$. The results obtained in Figure~\ref{fig:ALF-Y4-training-loss} show that with a fixed step size the lower order method is unable to achieve an accuracy better than $10^{-2}$ in training loss, whereas Y4 reaches accuracy $10^{-4}$. This illustrates what happens in the case of the adaptive time-stepping. A lower order method needs to reduce the step size to get to better accuracy. This implies more steps in the integration, and therefore, slower computations.}
\begin{figure} 
%\vspace{-2mm}
\centering
\includegraphics[width=0.35\textwidth]{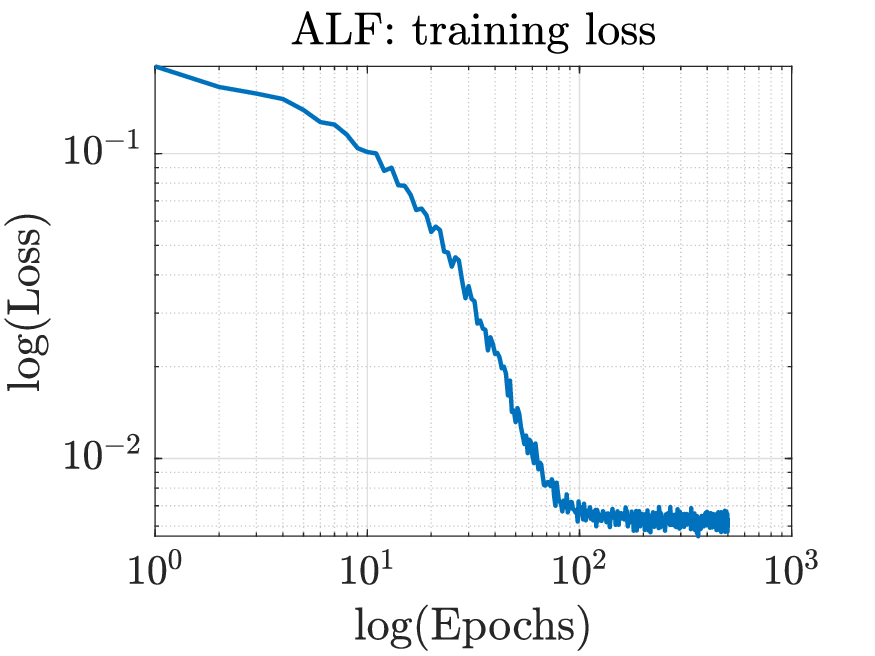} \ \ \ 
\includegraphics[width=0.35\textwidth]{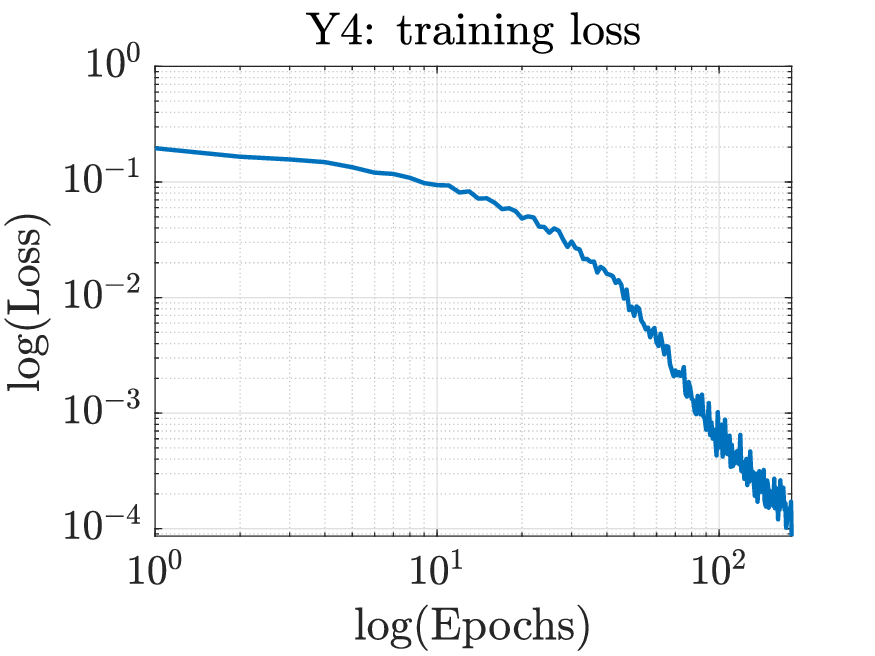}
%\fbox{\rule[-.5cm]{0cm}{4cm} \rule[-.5cm]{4cm}{0cm}}
\caption{Training loss is displayed in logarithmic scale for the parameter identification in case of coupled oscillation for ALF and Y4 with fixed step size $h=0.1$.}
\label{fig:ALF-Y4-training-loss}
%\vspace{-2mm}
\end{figure}
\subsubsection{Neural network parametrization}
The goal is to find the unknown potential governing \eqref{eq:coupled_osc.potential}. For this we assume a particular form of the potential, namely,
$$ V(q) = \sum_{i=1}^s\sum_{j=1}^n c_{i,j}\sigma_i(q_j) + \sum_{i=1}^d\sum_{j=1}^n\sum_{k=j+1}^n C_{i,j,k} \Sigma_i(\| q_j - q_k \|),$$
where $\sigma_i$ stand for different single particle potentials and $\Sigma_i$ for double particle potentials. In the case considered \textcolor{black}{above}, we have 
\begin{align*}
&c_{1,1} = \frac{a_1}{2}, \ c_{1,2} = \frac{a_2}{2}, \ \sigma_1(q) = q^2, \\
&c_{2,1} = \frac{b_1}{4}, \ c_{2,2} = \frac{b_2}{4}, \ \sigma_2(q) = q^4, \\
&C_{1,1,2} = \frac{e}{2}, \ \Sigma_1(x) = x^2.
\end{align*}
In the learning problem, we assume that functions $\sigma_1, \sigma_2$ and $\Sigma_1$ are unknown as well as parameters $a_1, a_2, b_1, b_2, e$. We parameterize the derivatives  $\frac{a_1}{2}\sigma_1', \frac{a_2}{2}\sigma_1',  \frac{b_1}{4}\sigma_2',\frac{b_2}{4}\sigma_2'$ and $\frac{e}{2} \Sigma_1'$ by neural networks each and use them to model the dynamics in \eqref{eq:coupled_osc.potential}. We use 5 neural network\chris{s}, which we denote by $\xi_1, \xi_2, \xi_3, \xi_4, \xi_5$. All of them have the same architecture $q \mapsto W_1\tanh{\left(W_2\tanh{(W_3 q)}\right)}$, where $W_1$ is a matrix of parameters of size $1 \times 100$, matrix $W_2$ is of size $100 \times 100$ and $W_3$ is of size $100 \times 1$. The resulting dynamics is defined by 
\begin{equation*}
\begin{aligned}
&\dot{q}_1 = v_1, \quad \dot{v}_1 = -\xi_1(q_1) -\xi_3(q_1) - \xi_5(q_1 - q_2), \\ % -  \frac{a_1}{2} \sigma_1'(q_1) - \frac{b_1}{4} \sigma_2'(q_1) - \Sigma_1'(q_1 - q_2), \\
&\dot{q}_2 = v_2, \quad \dot{v}_2 = -\xi_2(q_2) -\xi_4(q_2) - \xi_5(q_2 - q_1).% - \frac{a_2}{2} \sigma_1'(q_2) - \frac{b_2}{4} \sigma_2'(q_2) - \Sigma_1'(q_2 - q_1).  
\end{aligned}
\end{equation*} 
The equations parameterized by neural networks are then integrated using ALF or Yoshida composition of ALF2 at each epoch in the training. The training data is set to be a set of $1000$ trajectories with the initial conditions generated by the Halton sequence in a 4-dimensional box around $x_0 = (0.8,  -0.4, 0.0, 0.0)$ with diameter $2.0$. The optimizer is \texttt{AdamW} with initial learning rate $10^{-3}$ and scheduler \texttt{ExponentialLR} with $\gamma = 0.995$. In addition, we consider batches of 300 trajectories at each epoch with the resulting loss function of the same form as in the case of the parameter identification problem.  
%%%%%%%%%%%%%%%%%%%%%%%%%%%%%%%%%%%%%%%%%%%%%%%%%%%%%%%%%%%
\textcolor{black}{
\subsection{Discretized wave equations}\label{sec:app_discWave}
For generation of the training data, we consider the wave equation with potential $V(u) = \frac{1}{2}u^2$.
\textcolor{black}{The true motions can be expressed in the time-dependent Fourier series as
\[
u(t,x)=\sum_{m=-\infty}^\infty \hat u_m(t) e^{2 \pi i mx/L}, \quad L=1
\]
where the Fourier coefficients evolve as
\[
\hat u_m(t) = \gamma_m^{-1} \hat{v}_{m,0} \sin(\gamma_m t)+  \hat{u}_{m,0} \cos(\gamma_m t), \quad \gamma_m = \sqrt{1+\frac{4 \pi^2}{L^2}m^2}.
\]
Here $\hat{u}_{m,0}$, $\hat{v}_{m,0}$ are the Fourier coefficients of an initial wave $u(0,x)$ and velocity $u_t(0,x)$, respectively. Notice that a Fourier coefficient $\hat{u}_{m}(t)$ remains exactly zero over time if and only if  $\hat{u}_{m,0} =0=\hat{v}_{m,0}$.
Training data to initial data with only finitely many nonzero Fourier coefficients can, therefore, be obtained to machine precision by a spectral method. Alternatively, solutions can be computed by an application of the 5-point stencil as described in Example 7 (16) in \cite{Offen2024} on a fine mesh with discretization parameters $\Delta t = 1/160$, $\Delta x = 1/80$ and then subsampled to a mesh with $\Delta t = 1/40$, $\Delta x = 1/20$. In our case both methods yield the same training data up to a maximum error of order $1\mathrm{e}-4$.
In the training data creation, we sample initial $\hat{u}_{m,0}$, $\hat{v}_{m,0}$ from a standard normal distribution. It is then weighted by $e^{-4m^8}$ such that effectively only the first two Fourier modes are active. See figure \ref{fig:PDE.trainingData} for a plot of two of the solutions to the wave equation that were used to create the training data set.
}}
\begin{figure} 
\centering
\includegraphics[width=0.35\textwidth]{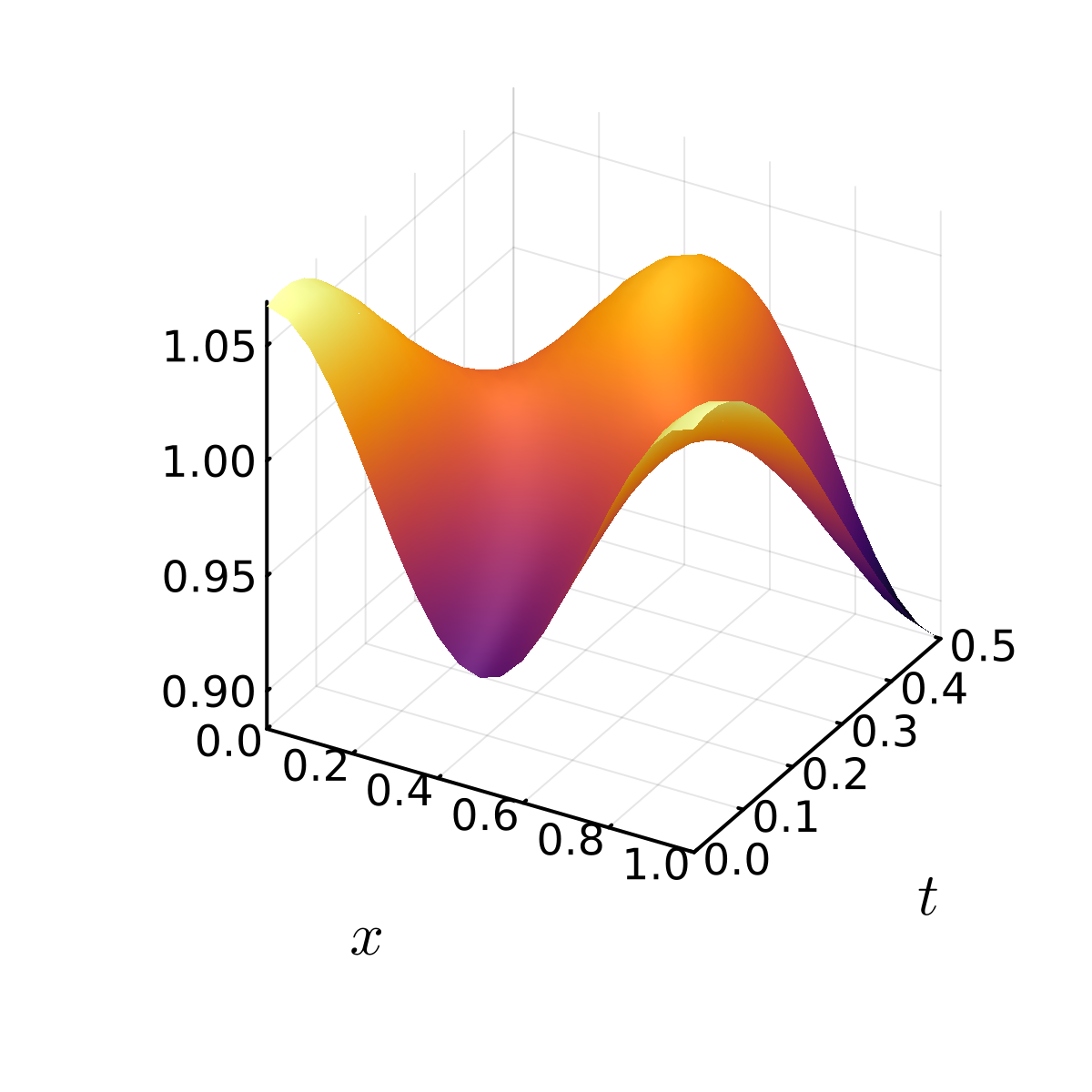} \ \ \ 
\includegraphics[width=0.35\textwidth]{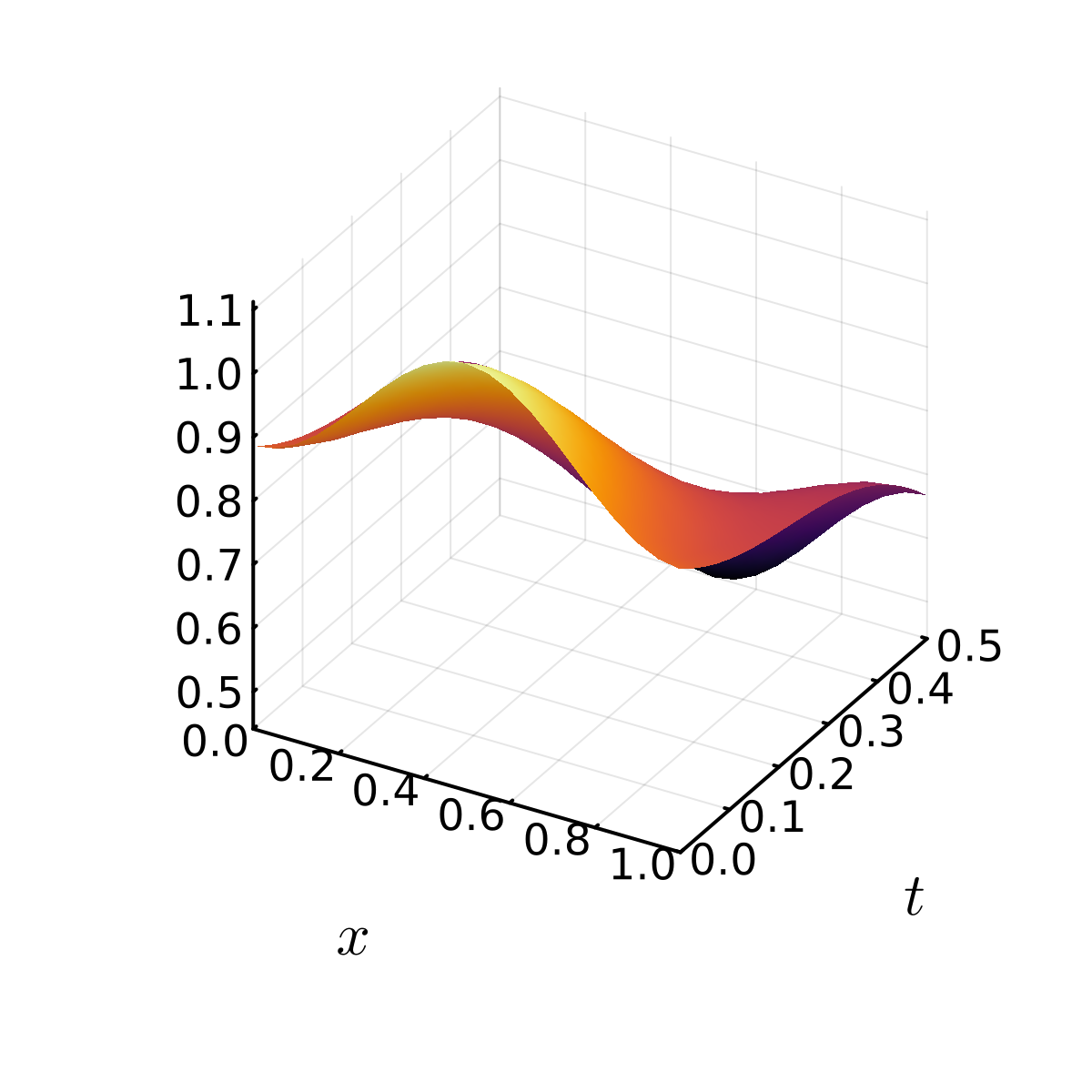} \\
\caption{Two samples of the training data used in section \ref{sec:app_discWave}.} 
\label{fig:PDE.trainingData}
%\fbox{\rule[-.5cm]{0cm}{4cm} \rule[-.5cm]{4cm}{0cm}}
\end{figure}
\textcolor{black}{
%by first discretizing $u(t,x)$ in spacial variable $x$ with $\Delta x = \frac{1}{80}$ leading to vector $u_d$, which approximates $u$ at the mesh points and then solving the corresponding system with step size $h = \frac{1}{160}$. The discretized PDE equation is an Euler-Lagrange equation for $L(u_d, (u_d)_t) = \frac{1}{2}\left( (u_d)t^\top (u_d)_t + u_d^\top A u_d - u_d^\top u_d \right)$ Then the solutions are sub-sampled to a mesh with $dx=\frac{1}{20}$ and the velocities are approximated using momenta $p_d = \frac{\partial L}{\partial \dot u}(u_d, \dot u_d) = u_d$ and the expression of the considered discrete Lagrangian.
In the training, we consider initial and final points of 50 trajectories on time interval $[0, 0.3]$ and 30 unseen trajectories in the testing. We use the optimiser LBFGS with the default values of the parameters. %The neural network approximating $f(u_d)$ in  \eqref{eq:disc:PDE} we use a composition of a linear block with dimensions $20 \times 100$, then a ReLU block of dimension $100 \times 100$ and then a linear block with dimensions $100 \times 20$.
In the numerical tests, we compare the behaviour of ALF, Y4 and Runge-Kutta 4(5). It can be seen in Figure~\ref{fig:PDE.training} that Y4 reaches the lowest values in the training loss faster than ALF. While RK45 is fastest, it also consumes more memory, which can make a crucial difference in high dimensional systems. We also report a lower time of computations per epoch for Y4 with respect to the results by ALF in Figure~\ref{fig:PDE.time}.}
\begin{figure} 
\centering
\includegraphics[width=0.35\textwidth]{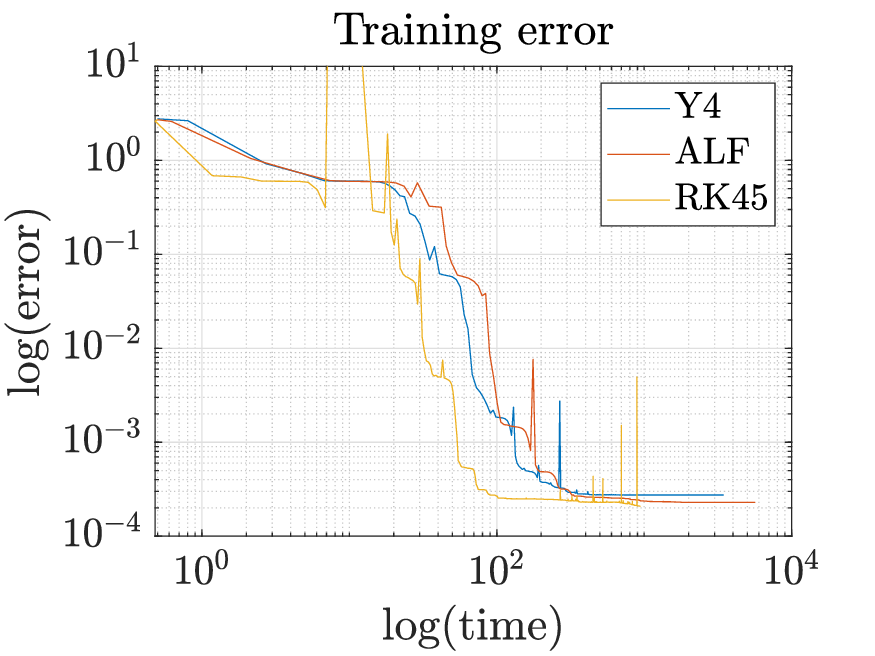} \ \ \ 
\includegraphics[width=0.35\textwidth]{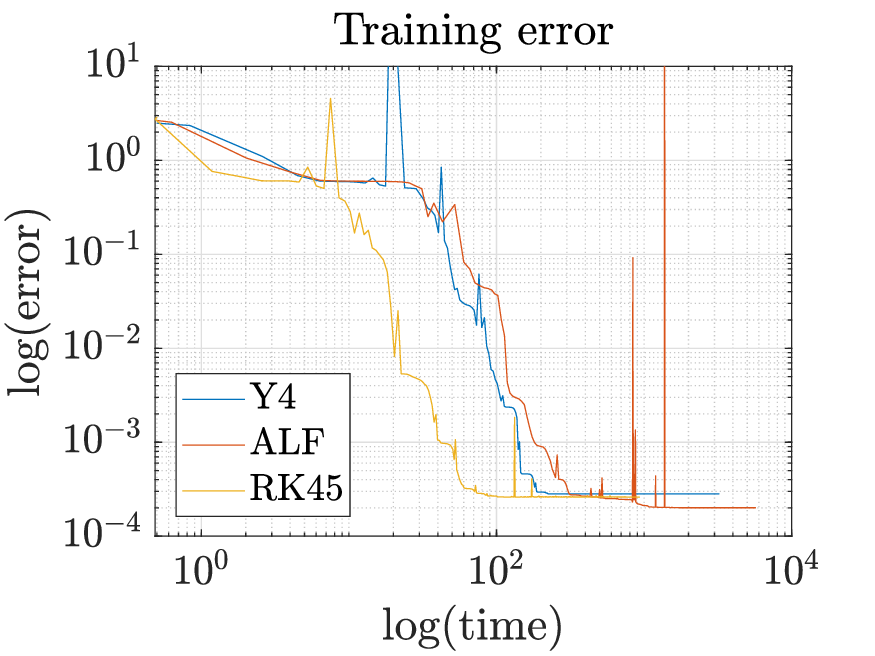} \\
\includegraphics[width=0.35\textwidth]{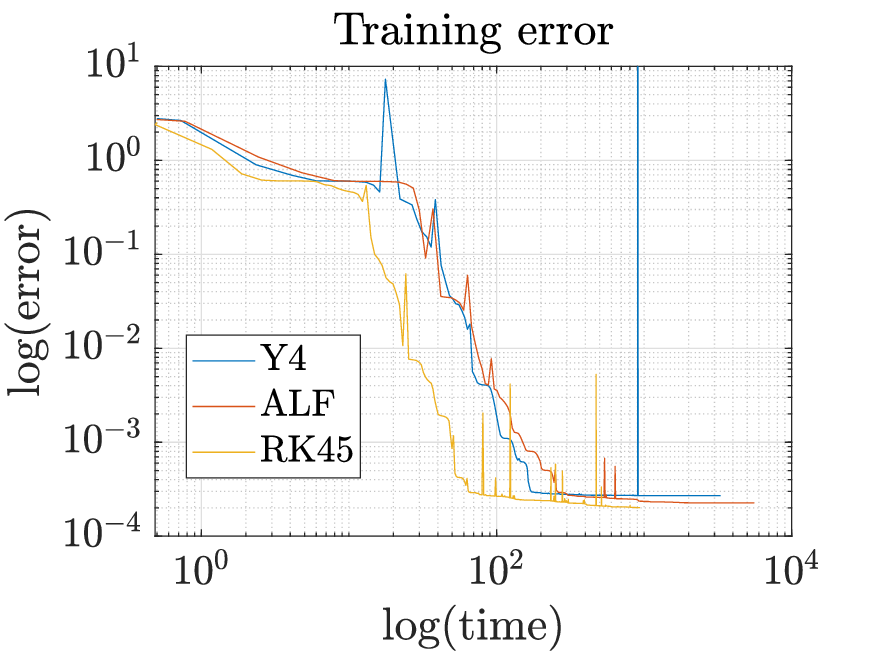} \ \ \ 
\includegraphics[width=0.35\textwidth]{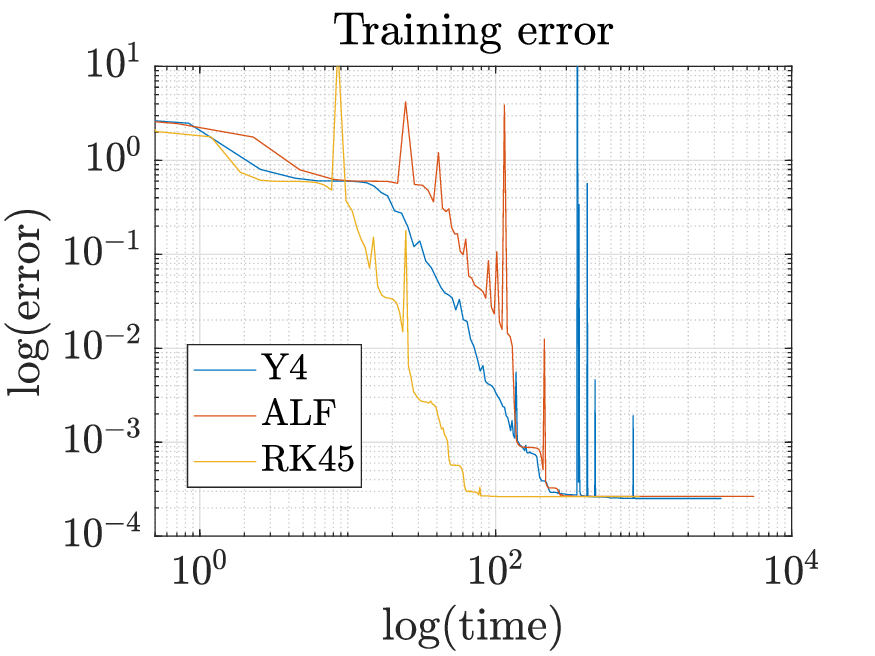}
\caption{Time of computation in function of epochs. Lower curve means faster computations.} 
\label{fig:PDE.training}
%\fbox{\rule[-.5cm]{0cm}{4cm} \rule[-.5cm]{4cm}{0cm}}
\end{figure}

%ffff
\begin{figure} 
\centering
\includegraphics[width=0.35\textwidth]{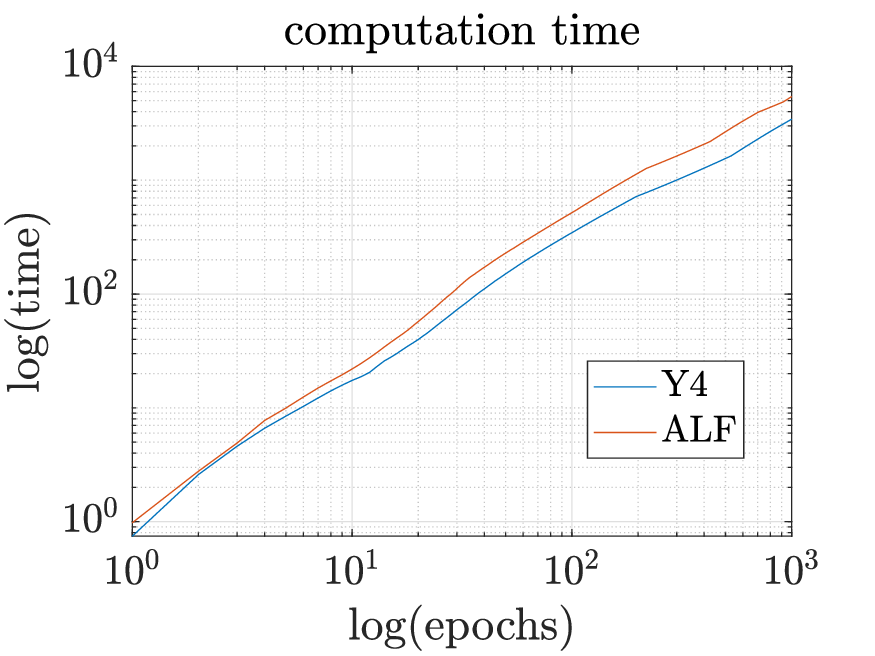} \ \ \ 
\includegraphics[width=0.35\textwidth]{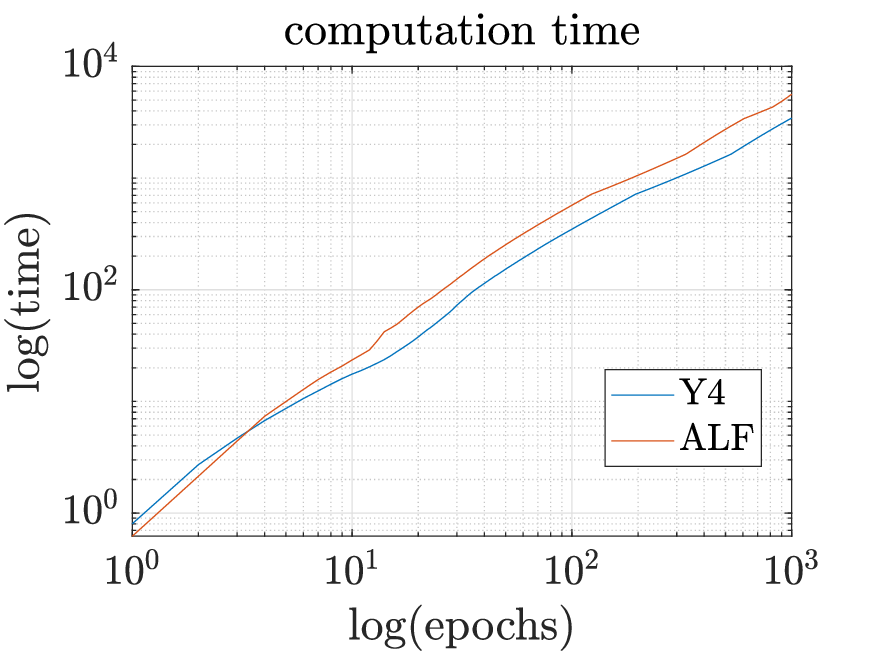} \\
\includegraphics[width=0.35\textwidth]{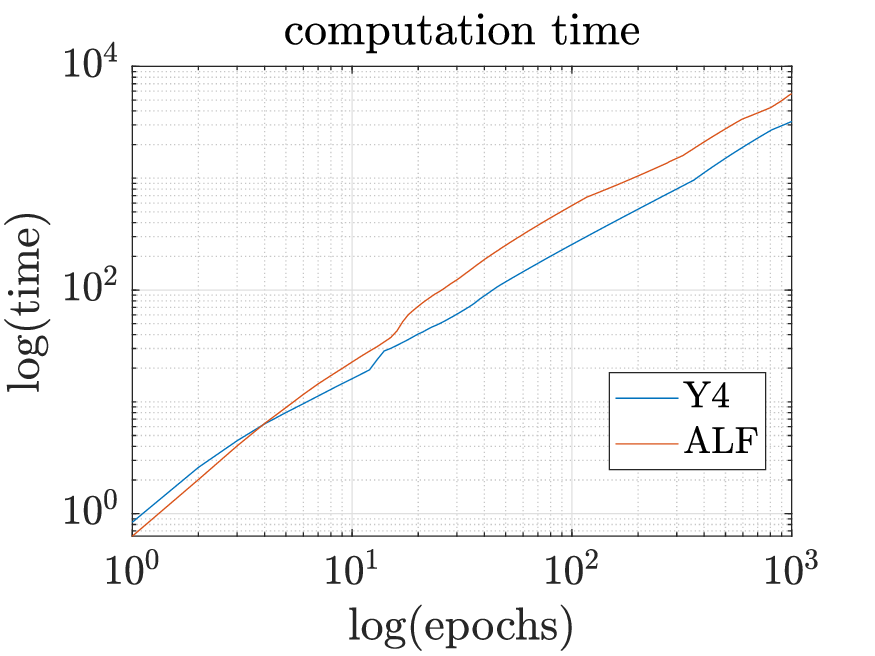} \ \ \ 
\includegraphics[width=0.35\textwidth]{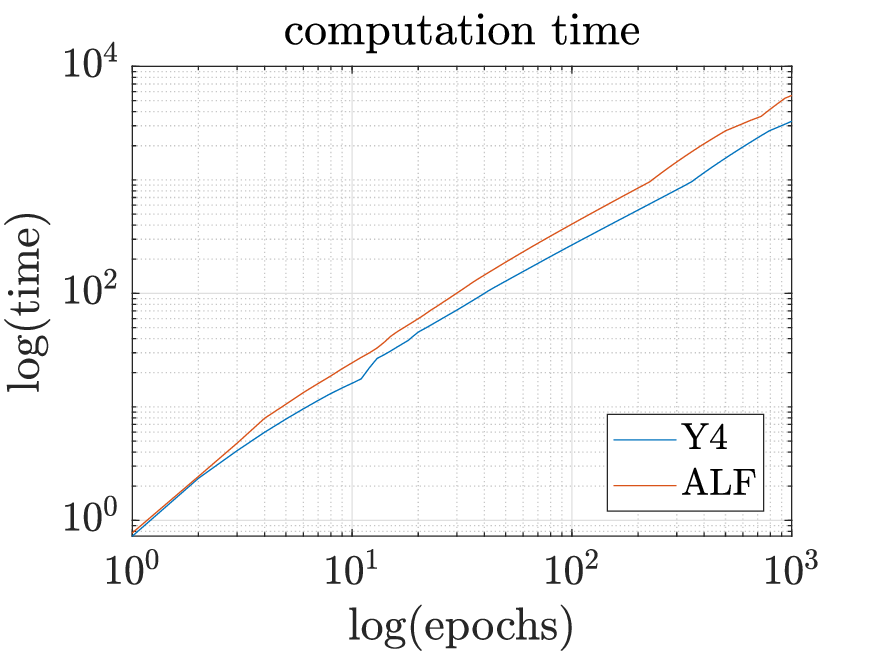}
\caption{Time of computation in function of epochs. When the curve is positioned lower, the corresponding algorithm is faster.} 
\label{fig:PDE.time}
%\fbox{\rule[-.5cm]{0cm}{4cm} \rule[-.5cm]{4cm}{0cm}}
\end{figure}

% \begin{table}
% %\vspace{-2mm}
%   \caption{Memory consumption by adaptive methods in discretized~PDE~example}
%   \label{table:PDENN}
%   \centering
% \begin{tabular}{ |p{6cm}||p{2.5cm}|p{2.5cm}|p{2.5cm}| }
% 	\hline
% 	\multicolumn{4}{|c|}{Full memory usage} \\
% 	\hline
% 	Random initialization of parameters in NN & adaptive ALF  &  adaptive Y4 & RK45\\
% 	\hline
% 	Initialization 1
%  &  228920K
%   &  226716K
%  &  246428K
%  \\
% 	\hline
% %	0.6   &  4.76 sec   & {1.89 sec}  \\
% %	\hline
% 	Initialization 2  &   231000K
%    &  228664K
%  & 370356K
%  \\
% 	\hline
% 	Initialization 3   &  229384K
%     &  226812K
%   &  233684K
% \\
% 		\hline
% 	Initialization 4   &  230948K
%   & 228720K
%  & 233640K
% \\
% %	\hline
% %	0.9&    5.65 sec   & { 1.25 sec } \\
% 	\hline
% 	Initialization 5  & 229084K
%  & 228728K
%  &233700K
% \\
% 	\hline
% \end{tabular}
% %\vspace{-2mm}
% \end{table}

\end{document}